\documentclass[10pt]{article}
\usepackage[utf8]{inputenc}
\usepackage[english]{babel}
\usepackage{amsmath}
\usepackage{amssymb}
\usepackage{amsfonts}
\usepackage{amsthm}
\usepackage{enumerate}
\usepackage{tikz-cd}
\usepackage{subfiles}
\usepackage[scr]{rsfso}
\usepackage{extarrows}
\usepackage{hyperref}
\usepackage{appendix}
\usepackage{bm}
\hypersetup{
    colorlinks=true,
    linkcolor=blue,
    filecolor=red,      
    urlcolor=red,
}
\title{Quantitative Gromov Compactness}
\author{Mohan Swaminathan}
\newtheorem{theorem}{Theorem}[section]
\newtheorem{lemma}[theorem]{Lemma}
\newtheorem{corollary}[theorem]{Corollary}
\newtheorem{proposition}[theorem]{Proposition}

\theoremstyle{definition}
\newtheorem{definition}[theorem]{Definition}
\theoremstyle{remark}
\newtheorem{remark}[theorem]{Remark}

\numberwithin{equation}{subsection}

\makeatletter
\@addtoreset{equation}{section}
\@addtoreset{theorem}{section}
\makeatother

\newcommand{\Mbar}{\overline{\mathcal M}}

\newcommand{\delbar}{\bar\partial}

\newcommand{\Aut}{\normalfont\text{Aut}}
\newcommand{\bA}{\mathbb{A}}

\newcommand{\bC}{\mathbb{C}}

\newcommand{\bP}{\mathbb{P}}

\newcommand{\bR}{\mathbb{R}}

\newcommand{\cA}{\mathcal{A}}

\newcommand{\cC}{\mathcal{C}}

\newcommand{\cF}{\mathcal{F}}

\newcommand{\cK}{\mathcal{K}}

\newcommand{\cM}{\mathcal{M}}

\newcommand{\cP}{\mathcal{P}}

\newcommand{\cT}{\mathcal{T}}

\newcommand{\fX}{\mathfrak{X}}

\newcommand{\fo}{\mathfrak{o}}

\newcommand{\fr}{\mathfrak{r}}

\newcommand{\dist}{\normalfont{\text{dist}}}
\begin{document}
	\maketitle
	%% ABSTRACT
	\begin{abstract}
	    We establish a quantitative version of the Gromov compactness theorem for closed genus $0$ pseudoholomorphic curves in the setting of a tamed almost complex manifold with bounded geometry.
	\end{abstract}
	%%%%%%%%%%%%%%%%%%%%%%%%%%%%%%%%%%%%%%%%%%%%%%%%%%%%
	%%%%%%%%%%%%%%%%% INTRODUCTION %%%%%%%%%%%%%%%%%%%%%
	%%%%%%%%%%%%%%%%%%%%%%%%%%%%%%%%%%%%%%%%%%%%%%%%%%%%
	\tableofcontents
	\begin{section}{Introduction}\label{intro}
	\noindent Let $(X,\omega,J,g)$ be a 4-tuple where $(X,g)$ is a smooth Riemannian manifold, $\omega$ is a symplectic form on $X$ and $J$ is an $\omega$-tame almost complex structure. We will assume that we have $C^k$ bounds on the geometry of $(X,\omega,J,g)$ for some integer $k\ge 3$. Gromov's compactness theorem (\cite[1.5.B]{gromov85}, \cite[Theorem 5.3.1]{McSa}, \cite[Theorem V.1.2]{hummel}) then implies that if we fix integers $\gamma,\ell\ge 0$, a constant $A>0$ and a compact subset $K\subset X$, then the moduli space 
	\begin{align}
	    \Mbar_{\gamma,\ell}(X,J;K)^{\le A}    
	\end{align}
	of (closed) stable $\ell$-pointed $J$-holomorphic maps in $X$ of genus $\gamma$ with energy $\le A$ and image lying in $K$ is compact. In this article, we will establish a quantitative version of this statement for $\gamma = 0$.
	%%%%%%%%%%%%%%%%%%%%%%%%%%%%%%%%%%%%%%%%%%%%%%%
	\subsection{Setup}
	We start by identifying $\bP^1$ with $\Mbar_{0,4}$ by mapping $[u:v]$ (with $u^3\ne v^3$) to the curve $\bP^1$ equipped with the four marked points $[1:1],[e^{2\pi i/3}:1],[e^{-2\pi i/3}:1],[u:v]$ and extending continuously. This endows $\Mbar_{0,4}$ with a natural round metric of total area $4\pi$ (see \textsection\ref{notation}). Next, for any $m\ge 4$ consider the map
	\begin{align}
	    \iota_m:\Mbar_{0,m}\to\Mbar_{0,4}^{\binom m4}
	\end{align}
	given by mapping $[\Sigma,x_1,\ldots,x_m]\in\Mbar_{0,m}$ to the $\binom m4$-tuple, indexed by choices $1\le i<j<k<l\le m$, of 4-point stabilizations $[\Sigma,x_i,x_j,x_k,x_l]^\text{st}\in\Mbar_{0,4}$. The map $\iota_m$ is a closed holomorphic embedding \cite[Theorem D.4.2]{McSa}. We can use the natural forgetful map $\pi_m:\Mbar_{0,m+1}\to\Mbar_{0,m}$, which is compatible with the embeddings $\iota_m$ and $\iota_{m+1}$, and the combinatorial identity $\binom {m+1}4 = \binom m4 + \binom m3$ to get the closed holomorphic embedding
	\begin{align}
	   j_{m+1}:\Mbar_{0,m+1}\hookrightarrow\Mbar_{0,m}\times Z_m 
	\end{align}
	which realizes $\Mbar_{0,m+1}$ as an algebraic family of closed subschemes of $Z_m := \Mbar_{0,4}^{\binom m3}$, parametrized by $\Mbar_{0,m}$. This algebraic family yields a natural map
	\begin{align}
	    \hat\iota_m:\Mbar_{0,m}&\to\cK(Z_m)\\
	    s&\mapsto\cC_s := \pi_m^{-1}(s)\subset Z_m
	\end{align}
	from the moduli space of stable $m$-pointed genus $0$ curves to the space of compact subsets of $Z_m$, endowed with the Hausdorff distance metric (Definition \ref{hdorff-defn}). Here, we view $Z_m$ as a metric space by equipping it with the distance function obtained by taking the maximum of the distance functions on each of the $\Mbar_{0,4}$ factors. It is easy to verify that $\hat\iota_m$ is injective and continuous and thus, a closed topological embedding (since the domain is compact and the target is Hausdorff).
	\begin{definition}
	    For any $L>0$ and any integer $m\ge 3$, define the \textbf{\emph{$L$-distance}} on $Z_m\times X$ to be the metric which is given by the maximum of the distance function on $Z_m$ and $L^{-1}$ times the distance function on $X$. For compact subsets $A,B\subset Z_m\times X$, let $d_{H,L}(A,B)$ the Hausdorff $L$-distance between $A$ and $B$. Given $p,q\in\Mbar_{0,m}$ and continuous maps
	    $f:\cC_p\to X,g:\cC_q\to X$, define the \textbf{\emph{$L$-distance}} between $f$ and $g$ to be
	    \begin{align}
	        \dist_L(f,g) = d_{H,L}(\Gamma_f,\Gamma_g)
	    \end{align}
	    where $\Gamma_f$ (resp. $\Gamma_g$) is the graph of $f:\cC_p\to X$ (resp. $g:\cC_q\to X$), viewed as a compact subset of $Z_m\times X$.
	\end{definition}
	\begin{definition}
	    Given $A>0$, an integer $\ell\ge 0$, a stable $J$-holomorphic map $u = [\Sigma,x_1,\ldots,x_\ell,f]\in\Mbar_{0,\ell}(X,J;K)^{\le A}$ and a nonnegative integer $m\ge 3-\ell$, define an \textbf{\emph{$m$-decoration}} of $u$ to be a choice a point $p\in\Mbar_{0,\ell + m}$ and a biholomorphism $\varphi:\Sigma\to\cC_p$ with $\varphi(x_1),\ldots,\varphi(x_\ell)$ being the first $\ell$ points, in this order, on the $(m+\ell)$-pointed curve $\cC_p$. An $m$-decoration $\varphi$ yields a $J$-holomorphic map $u^\varphi := u\circ\varphi^{-1}$ which has a well-defined associated \textbf{$\varepsilon$-local Lipschitz constant}
	    \begin{align}
	        \text{Lip}_\varepsilon(u^\varphi) = \sup_{0<\dist(q_1,q_2)\le\varepsilon} \frac{\dist(u^\varphi(q_1),u^\varphi(q_2))}{\dist(q_1,q_2)}
	    \end{align}
	    for every $0<\varepsilon\le\pi$. Here, we are using metric induced by the inclusion $\cC_p\subset Z_{\ell + m}$ on the domain. (One should think of a local Lipschitz estimate as similar to a local gradient bound.) Moreover, given any $\delta>0$, we can define the neighborhood
	    \begin{align}
	        V_{L,\delta}(u;\varphi)\subset\Mbar_{0,\ell}(X,J;K)^{\le A}
	    \end{align}
	    of $u$ to consist of all those stable $J$-holomorphic maps $v\in\Mbar_{0,\ell}(X,J;K)^{\le A}$ for which there exists an $m$-decoration $\psi$ satisfying $\dist_L(u^\varphi,v^\psi)\le\delta$.
	\end{definition}
	%%%%%%%%%%%%%%%%%%%%%%%%%%%%%%%%%%%%%%%%%%%%%%%%
	\subsection{Main result}
	We can now state the main result, proved in \textsection\ref{last-proof}. Much more precise versions of this statement are contained in Theorems \ref{holo-map-tot-bound} and \ref{main-covering-theorem}.
	\begin{theorem}\label{intro-theorem}
	    Assume $(X,\omega,J,g)$ has bounded geometry, $K\subset X$ is a compact subset, $\ell\ge 0$ in an integer and $A>0$ is a real number. Then, there exists $\lambda_0>0$ depending only on the bounds on the geometry, an integer $m\ge 0$ and a real number $\Lambda>0$ both depending explicitly on $\ell, A$, the local bounds on the geometry of $X$ and not on $K$ with the following significance.
	    \begin{enumerate}[\normalfont(i)]
	        \item Given any $u\in\Mbar_{0,\ell}(X,J;K)^{\le A}$, it has an $m$-decoration $\varphi$ such that the associated map $u^\varphi$ has $1$-local Lipschitz constant $\le \Lambda\lambda_0$.
	        \item Let $u\mapsto\varphi_u$ denote an assignment of a $m$-decoration of $1$-local Lipschitz constant $\le\Lambda\lambda_0$ as in {\normalfont(i)} to each $u\in\Mbar_{0,\ell}(X,J;K)^{\le A}$. Given any positive $0<\delta\le 1$, we can then find an integer $N\ge 0$ depending explicitly on $\ell, A$, $K\subset (X,\omega,J,g)$ and $\delta$ with the following property. There is a subset $I\subset\Mbar_{0,\ell}(X,J;K)^{\le A}$ with $\le N$ elements, such that for every $u\in\Mbar_{0,\ell}(X,J;K)^{\le A}$, we can find a corresponding $v\in I$ such that we have
	        \begin{align}
	            \dist_{\lambda_0} (u^{\varphi_u},v^{\varphi_v})\le4\delta.
	        \end{align}
	        In particular, the neighborhoods $\{V_{\lambda_0,4\delta}(v;\varphi_v)\}_{v\in I}$ cover $\Mbar_{0,\ell}(X,J;K)^{\le A}$.
	    \end{enumerate}
	    In fact, we can take $m$ and $\Lambda$ to be given by the explicit formulas
	    \begin{align}
	        m &= \lfloor c\cdot(\ell + A/\lambda^2) \rfloor \\
	        \log\Lambda &= c\cdot(\ell + A/\lambda^2)
	    \end{align}
	    where $\lambda$ is a constant depending only on the local bounds of the geometry of $X$ while $c$ is a sufficiently large absolute constant. Having fixed $m$ and $\Lambda$, we can take $N$ to be given by the following explicit formula
	    \begin{align}
	        N = \left(1 + \sigma\delta^{-2k}\nu(K,\lambda_0)\right)^{(8\pi\Lambda^2\delta^{-2})^{\binom{m+\ell}3}}.
	    \end{align}
	    where $\sigma$ is a constant depending on the local bounds on the geometry of $X$, $\nu(K,\lambda_0)$ is the size of the smallest $\lambda_0$-net in the compact metric space $K$ and $2k = \dim X$.
	\end{theorem}
	\begin{remark}
	    Part (i) of Theorem \ref{intro-theorem} (or, more precisely, Theorem \ref{main-covering-theorem}) is analogous to the result of Groman on thick-thin decompositions of holomorphic curves \cite{groman}. Part (ii) provides a way of using part (i) to estimate the ``size" or ``complexity" of the moduli space, provided we can understand the ``complexity" of each of the individual neighborhoods $V_{\lambda_0,\delta}(v;\varphi_v)$, each of which can be covered by an explicit Kuranishi chart if $\delta$ is sufficiently small (how small should depend on $A$ and the bounds on the geometry).
	\end{remark}
	\begin{remark}
	    Theorem \ref{intro-theorem}(i) follows from a careful reworking of the usual proof of convergence modulo bubbling in \cite[Chapter 4]{McSa} for a sequence of maps with bounded energy. The new marked points are chosen so that they have the effect of magnifying regions where energy is concentrating and this brings down the Lipschitz constant. Theorem \ref{intro-theorem}(ii) follows by examining the proof of the Arzel\`a-Ascoli Theorem for Lipschitz maps and modifying it to deal with the case of variable domains (a precise form of this argument appears in Theorem \ref{holo-map-tot-bound}).
	\end{remark}
	%%%%%%%%%%%%%%%%%%%%%%%%%%%%%%%%%%%%%%%%%%%%%%%%%
	\subsection{Acknowledgements}
	I would like to thank my advisor John Pardon for several helpful discussions and for comments on an earlier version of this article. I am also grateful to Shaoyun Bai, Yash Deshmukh and Thomas Massoni for useful conversations.
	\end{section}
	%%%%%%%%%%%%%%%%%%%%%%%%%%%%%%%%%%%%%%%%%%%%%%%%%%
	%%%%%%%%%%%%%%% PREPARATION %%%%%%%%%%%%%%%%%%%%%%
	%%%%%%%%%%%%%%%%%%%%%%%%%%%%%%%%%%%%%%%%%%%%%%%%%%
	\begin{section}{Preparation}\label{prep}
	Before beginning the main discussion, we set notations and recall some fundamental analytical properties of pseudoholomorphic curves.
	%%%%%%%%%%%%%%%%%%%%%%%%%%%%%%%%%%%%%%%%%%%
	\subsection{Notation}\label{notation}
	For $r\ge 0$, an integer $n\ge 1$ and $x\in\bR^n$, define $B^n(x,r)$ to be the set of $y\in\bR^n$ satisfying $|y-x|\le r$. For $x = 0$, we will denote $B^n(x,r)$ simply by $B^n(r)$. Denote the boundary spheres of $B^n(x,r)$ and $B^n(r)$, by $S^{n-1}(x,r)$ and $S^{n-1}(r)$ respectively.\\\\
	\noindent For a set $\Omega\subset\bR^n$ and a point $z\in\bR^n$, we define the distance from $z$ to $\Omega$ to be the quantity 
	\begin{align}\label{point-set-distance}
	    |z - \Omega|= \inf_{w\in\Omega}|z - w|.
	\end{align}
	We will encounter many constants in our estimates below and while their exact values are not important, we will find it useful to record their dependence on parameters (e.g. the bounds on the geometry of $X$ and previously defined constants). We indicate that a constant $C$ explicitly depends on the $C^k$ bounds on the geometry of $X$ and some other parameters $a_1,\ldots,a_N$ by writing
	\begin{align}
	    C = C(X_k,a_1,\ldots,a_N)
	\end{align}
	when it is first defined.
	\\\\
	\noindent Consider the complex projective line $\bP^1$ with homogeneous coordinates $[x:y]$. The standard round metric on $\bP^1$ (with area $4\pi$) is the K\"ahler metric with associated K\"ahler form {$\omega_{\bP^1}$} given (in the standard complex affine coordinate $z = x/y$) by
	    \begin{align}
	        \omega_{\bP^1} = \frac{2i\cdot dz\wedge d\bar z}{(1 + |z|^2)^2}
	   \end{align}
	    Similarly, define the standard flat metric on $\bC$ to be the K\"ahler metric with K\"ahler form {$\omega_\bC$} given (in the standard complex coordinate $z$) by
	    \begin{align}
	        \omega_\bC = \textstyle\frac i2\cdot dz\wedge d\bar z
	    \end{align}
	\noindent We use the notation $\ell(\cdot)$ to denote the $g$-length of paths in $X$. We use $\cA(\cdot)$ to denote the $\omega$-action of a loop in $X$ (assumed to be sufficiently small so that the symplectic action is defined unambiguously). Finally, given a compact Riemann surface (or, more generally, a projective prestable curve) $\Sigma$ and a smooth map $u:\Sigma\to X$, we denote its energy by
	\begin{align}
	    E(u) = \frac12\int_\Sigma|du|^2_h\,\text{dvol}_h
	\end{align}
	where $h$ is any choice of a conformal metric on $\Sigma$. For an open subset $U\subset\Sigma$, we sometimes write $E(u,U)$ to denote the energy of the map $u|_U:U\to X$. For any two points $x,x'\in X$, we use the notation $\dist(x,x')$ to denote the $g$-distance between $x$ and $x'$. We will also use the symbol $\dist(\cdot,\cdot)$ to denote distances in any metric space when the choice of the metric is clear from the context.
	%%%%%%%%%%%%%%%%%%%%%%%%%%%%%%%%%%%%%%%%%%%%%%
	\subsection{Local \emph{a priori} estimates}
	We note the following fundamental estimates for pseudoholomorphic maps into $X$.
	\begin{lemma}[Elliptic bootstrapping]\label{ell-boot}
	    There exist positive constants $\lambda_0 = \lambda_0(X_1)$ and $c_k = c_k(X_{k+1})\ge 1$ for each integer $k\ge 2$, with the following significance. If $u:B^2(r)\to X$ is a $J$-holomorphic map, then
	    \begin{align}
	        |(\nabla^k u)(0)|\le c_k\|du\|_\infty\left(\frac1r + \frac{\|du\|_\infty}{\lambda_0}\right)^{k-1}.
	    \end{align}
	\end{lemma}
	\begin{proof}
	    See Appendix \ref{ell-boot-appx} for the proof.
	\end{proof}
	\begin{lemma}[Mean value inequality]\label{mvi}
	    There exist positive constants $C\ge 1$ and $q = q(X_2)$ with the following significance. If $u:B^2(r)\to X$ is a $J$-holomorphic map with
	    $E(u) \le q$, then we have
	    \begin{align}\label{mv-est}
	        r^2|du(0)|^2\le C\cdot E(u).
	    \end{align}
	\end{lemma}
	\begin{proof}
	    This is just Lemma 4.3.1(i) from \cite{McSa}. See Appendix \ref{alt-mvi} for an alternate proof.
	\end{proof}
	\begin{lemma}[Long cylinders of small energy]\label{exp-decay}
	    There exist positive constants $c = c(X_2)$ and $l = l(X_2)\le\lambda_0$, with the following significance. Suppose $u:[R_-,R_+]\times S^1\to X$ is a $J$-holomorphic map with $\|du\|_\infty\le l$. For each $s\in[R_-,R_+]$, let $\gamma_s:S^1\to X$ be the loop $t\mapsto u(s,t)$. Then, for all $R_-\le a\le b\le R_+$, we have
	    \begin{align}\label{action-diff}
	        \int_{[a,b]\times S^1}u^*\omega = \cA(\gamma_a)-\cA(\gamma_b).
	    \end{align}
	    For each $s\in[R_-,R_+]$, we also have 
	    \begin{align}\label{weak-iso-peri}
	        |\cA(\gamma_s)|\le c\cdot\ell(\gamma_s)^2.  
	    \end{align}
	    Moreover, there exist constants $K_k = K_k(X_{k+1})\ge 1$ for each integer $k\ge 1$ such that the estimate
	    \begin{align}\label{exp-decay-estimate}
	        |(\nabla^k u)(s,t)|\le K_k\sqrt{E(u)}(e^{s-R_+} + e^{R_--s})
	    \end{align}
	    holds for all $(s,t)\in[R_-+1,R_+-1]\times S^1$, provided that $R_+-R_-\ge 2$.
	\end{lemma}
	\begin{proof}
	    This is a variant of previously known statements. See \cite[Lemma 11.2]{FO-kuranishi} and \cite[Lemma 4.7.3]{McSa} for example. For the proof, see Appendix \ref{exp-decay-appx}.
	\end{proof}
	\end{section}
	%%%%%%%%%%%%%%%%%%%%%%%%%%%%%%%%%%%%%%%%%%%%%%%%%%
	%%%%%%%%%%%%%% FAMILIES %%%%%%%%%%%%%%%%%%%%%%%%%%
	%%%%%%%%%%%%%%%%%%%%%%%%%%%%%%%%%%%%%%%%%%%%%%%%%%
	\begin{section}{Compact families of curves and holomorphic maps}\label{families}
	In this section, we study the compactness properties of certain explicit families of curves and stable maps defined on them. The curves in such a family will be modeled on a fixed marked rooted tree. Please consult Appendix \ref{tree-basics} for the notation and relevant background on trees which will be assumed below.
	%%%%%%%%%%%%%%%%%%%%%%%%%%%%%%%%%%%%%%
	\subsection{Curves}\label{curve-family}
	\begin{definition}\label{curve-assoc-tree}
	    Fix a rooted tree $\cT = (V,E,\partial)$ and let $e_0$ be its root edge and $v_0$ the corresponding root vertex (see Definitions \ref{tree-def} and \ref{rooted-tree-def}). Define $\bA_\cT$ to be the vector space
	    \begin{align}
	        \bA_\cT = \bC^{E_\text{int}}\oplus\bigoplus_{v\in V}(\bC^2)^{v^+}
	    \end{align}
	    where $E_\text{int}\subset E$ is the set of full edges. Let the coordinate functions corresponding to the $\bC^{E_\text{int}}$ factor be denoted by $\{\gamma_e\}$ while the coordinates corresponding to an edge $e\in v^+$ for any $v\in V$ (see Definition \ref{edge-orient-def} and Lemma \ref{orient-rooted-tree}) are denoted by $\{(z_{v,e},\rho_{v,e})\}$. Given a positive path (see Definition \ref{path-def}) $\cP = (v_1,\ldots,v_n)$ between vertices $u = v_1$ and $v = v_n$ and an edge $e\in v^+$, define the polynomial function $z_{v,e}^u:\bA_\cT\to\bC$ as
	    \begin{align}\label{marked-point-across-components}
	        z^u_{v,e} = \sum_{k=1}^n \left(\textstyle\prod_{1\le j<k}\gamma_{e_j}\rho_{v_j,e_j}\right)z_{v_k,e_k}
	    \end{align}
	    where we define $e_i$ be the edge with endpoints $v_i$ and $v_{i+1}$ (for $1\le i<n$) and $e_n = e$. Observe that for any vertex $v$ and any edge $e\in v^+$, we have $z_{v,e}^v = z_{v,e}$. 
	    %\textcolor{red}{Now, define $\cM_\cT\subset\bA_\cT$ to be the Zariski open (affine) subset given by the conditions \eqref{points-dont-collide} and \eqref{finite-rescaling} below. Whenever $e\ne e'$ (and both $\ne e_0$), we require
	    %\begin{align}\label{points-dont-collide}
	    %    z^u_{v,e}\ne z^u_{v',e'}
	    %\end{align}
	    %where $u$ is the nearest common ancestor (in the sense of Definition \ref{common-ancestor}) of $v = e^-,v' = e'^-$. We also require
	    %\begin{align}\label{finite-rescaling}
	    %    \rho_{v,e}\ne 0
	    %\end{align}
	    %for each $e\in v^+$ corresponding to each $v\in V$.} 
	    Define $\cM_\cT\subset\bA_\cT$ to be the Zariski open subset where the polynomial function $F_\cT:\bA_\cT\to\bC$, defined by
	    \begin{align}\label{function-def-M}
	        F_\cT = \prod(z^u_{v,e} - z^u_{v',e'})\cdot\prod\rho_{u,e}
	    \end{align}
	    is nonzero. Here, the first product is over all pairs $\{e,e'\}$ of distinct edges in $E\setminus\{e_0\}$ (with $v = e^-$, $v' = e'^-$ and $u$ being the nearest common ancestor of $v,v'$, in the sense of Definition \ref{common-ancestor}) while the second product is over all $e\in E_\text{int}$ (with $u = e^-$). Next, define the closed subscheme {$\cC_\cT$} $\subset\cM_\cT\times(\bP^1)^V$
	    to consist of the tuples $\{\gamma_e\},\{(z_{v,e},\rho_{v,e})\},\{[x_v:y_v]\}$ which, for each edge $e\in E_\text{int}$ (with $u = e^-$ and $v = e^+$), satisfy the homogeneous equation
	    \begin{align}\label{def-eqn}
	        (x_u - z_{u,e}y_u)y_v = \gamma_e\rho_{u,e}x_vy_u.
	    \end{align}
	    Define the projection map {$\pi_\cT:\cC_\cT\to\cM_\cT$} in the obvious fashion by dropping the $(\bP^1)^V$ factor. We define $\pi_\cT:\cC_\cT\to\cM_\cT$ to be the \emph{\textbf{family of curves associated to $\cT$}}. For any point $p\in\cM_\cT$, we define {$\cC_{p,\cT}$} to be the fibre of $\pi_\cT$ over $p$. For each $v\in V$, we have a natural coordinate projection $\pi_v:\cC_\cT\to\bP^1$ (got by projecting to $(\bP^1)^V$ and then projecting to the $\bP^1$ factor corresponding to $v$).
	\end{definition}
	\begin{lemma}\label{prop-family-curve}
	    Let $p^*= (\{\gamma_e^*\},\{(z_{v,e}^*,\rho_{v,e}^*)\})\in\cM_\cT$ be any point. Then, the fibre $\cC_{p^*,\cT}$ is naturally regarded as a closed subscheme of $(\bP^1)^V$.
	    \begin{enumerate}[\normalfont(i)]
	        \item If $\gamma_e^*\ne 0$ for all $e\in \normalfont E_\text{int}$, then $\pi_v:\cC_{p^*,\cT}\to\bP^1$ is an isomorphism for all $v\in V$.
	        
	        \item The projection $\pi_{v_0}:\cC_{p^*,\cT}\to\bP^1$ is a local isomorphism over the point $[1:0]\in\bP^1$.
	        
	        \item Let $e_0\ne f\in\normalfont E_\text{ext}$ be a half edge and set $w = f^-$. Then, the projection $\pi_w:
	        \cC_{p^*,\cT}\to\bP^1$ is a local isomorphism over the point $[z^*_{w,f}:1]\in\bP^1$.
	        
	        \item Consider $e\in\normalfont E_\text{int}$ such that $\gamma^*_e = 0$. Write $u = e^-$ and $v = e^+$ and let $\cT_u$, with vertex set $u
	        \in V_u$, and $\cT_v$, with vertex set $v\in V_v$, be the trees obtained by splitting $\cT$ along $e$ in the sense of Definition \ref{splitting-def}. Let $p^*_u$, resp. $p^*_v$, be the projections of $p^*$ to $\cM_{\cT_u}$, resp. $\cM_{\cT_v}$. Then, there exist (unique) points 
	        \begin{align}
	           q_{p^*,u}\in\cC_{p^*_v,\cT_v}\subset(\bP^1)^{V_v}\\
	           q_{p^*,v}\in\cC_{p^*_u,\cT_u}\subset(\bP^1)^{V_u}
	        \end{align}
	        such that we have $\pi_v(q_{p^*,u}) = [1:0]$, $\pi_u(q_{p^*,v}) = [z^*_{u,e}:1]$ and, moreover, for these points we have
	        \begin{align}\label{break-into-comps}
	            \cC_{p^*,\cT} = (\cC_{p^*_u,\cT_u}\times q_{p^*,u})\cup(q_{p^*,v}\times\cC_{p^*_v,\cT_v})\subset(\bP^1)^V
	        \end{align}
	        where the union is scheme-theoretic. Moreover, $\cC_{p^*,\cT}$ is isomorphic to a prestable genus $0$ curve.
	        
	        \item We have an closed algebraic embedding $\sigma_\cT:\normalfont E_\text{ext}\times\cM_\cT\to\cC_\cT$, which commutes with the projection to $\cM_\cT$ and has the following property. We have $\pi_v(\sigma_\cT(e_0,\cdot))\equiv[1:0]$ for all $v\in V$, while for $e_0\ne e\in \normalfont E_\text{ext}$ with $u = e^-$, we have $\pi_u(\sigma_\cT(e,\cdot)) \equiv [z_{u,e}(\cdot):1]$. Moreover, the image of any point under $\sigma_\cT$ lies in the smooth locus of the corresponding fibre of $\pi_\cT$.
	        
	    \end{enumerate}
	\end{lemma}
	\begin{proof}
	    See Appendix \ref{curve-family-proofs} for the proof.
	\end{proof}
	\begin{lemma}
	    The map $\pi_\cT:\cC_\cT\to\cM_\cT$ of Definition \ref{curve-assoc-tree} and $\sigma_\cT$ of Remark \ref{prop-family-curve}{\normalfont(v)} together define a proper, flat algebraic family of prestable genus $0$ curves with $\normalfont|E_\text{ext}|$ marked points. This family is a versal deformation of each of its fibres. If $\cT$ is a stable tree, then the fibres of this family are stable curves.
	\end{lemma}
	\begin{proof}
	    We have already shown that for any $p^*\in\cM_\cT$, the fibre $\cC_{p^*,\cT}$ is a prestable genus $0$ and that $\{\sigma_\cT(e,p^*)\}_{e\in E_\text{ext}}$ is a collection of $|E_\text{int}|$ distinct smooth points of $\cC_{p^*,\cT}$. Moreover, it is also clear that the stability of $\cT$ implies the stability of $\cC_{p^*,\cT}$ equipped with the marked points given by $\sigma_\cT$. That the family is algebraic and proper is obvious from the definition and its flatness is a result of applying \cite[Chapter X, Proposition 2.1]{ACGH2} to the defining equations \eqref{def-eqn}. Versality follows since we are in genus $0$ and the coordinates $\gamma_e$ allow for smoothing of nodes and the coordinates $z_{u,e}$ allow for variation of marked points. Stability of $\cC_{p,\cT}$ for $p\in\cM_\cT$ is a immediate if $\cT$ is a stable tree.
	\end{proof}
	\begin{definition}\label{metric-embedding}
	    Define the space 
	    \begin{align}
	        \bP_\cT=\prod_{v\in V}(\bP^1\times(\bP^1)^{v^+})
	    \end{align}
	    which is naturally endowed with obvious projections $\pi_{v,e}:\bP_\cT\to\bP^1$ for each pair $v\in V$ and $e\in E$ with $v\in\partial(e)$. Define a K\"ahler form and distance function on $\bP_\cT$ by
	    \begin{align}
	        \omega_{\bP_\cT} &= \sum_{v,e}\pi_{v,e}^*\omega_{\bP^1}\\
	        \dist(q,q') &= \max_{v,e}\dist(\pi_{v,e}(q),\pi_{v,e}(q')).
	    \end{align}
	    Pullback the distance function induced by the $\|\cdot\|_\infty$-norm (resp. the standard K\"ahler form) on $\bA_\cT$ to get a distance function (resp. a K\"ahler form $\omega_{\cM_\cT}$)  on $\cM_\cT$. %Here, $F_\cT$ is as in \eqref{function-def-M}
	\end{definition}
	\begin{definition}\label{proj-embedding}
	    Let $p^* = (\{\gamma^*_e\},{(z^*_{v,e},\rho^*_{v,e})})\in\cM_\cT$ be any point. Given any vertex $u\in V$ and edge $e\in u^+$, define the biholomorphism $\varphi_{p^*,u,e}\in\Aut(\bP^1)$ by the formula
	    \begin{align}
	        \varphi_{p^*,u,e}([x_u:y_u]) = [x_u - z^*_{u,e}y_u:\rho^*_{u,e}y_u].
	    \end{align}
	    Now, define the embedding $\iota_{p^*,\cT}:\cC_{p^*,\cT}\to\bP_\cT$ by the formula
	    \begin{align}
	        \{[x_u:y_u]\}_{u\in V}\mapsto\{([x_u:y_u],\{\varphi_{p^*,u,e}([x_u:y_u])\}_{e\in u^+})\}_{u\in V}.
	    \end{align}
	    Putting these together for all $p^*$, we get an algebraic embedding
	    \begin{align}
	        \iota_\cT:\cC_\cT&\to\cM_\cT\times\bP_\cT\\
	        (p^*,q)&\mapsto(p^*,\iota_{p^*,\cT}(q)).
	    \end{align}
	    Using this embedding, we endow $\cC_\cT$ with a distance function by taking the maximum of the distance functions on $\cM_\cT$ and $\bP_\cT$. Similarly define the K\"ahler form $\omega_{\cC_\cT}$ by pulling back $\omega_{\cM_\cT}\oplus\omega_{\bP_\cT}$ via $\iota_\cT$.
	\end{definition}
	%\begin{remark}\label{complete-metric}
	%    Note that since $\iota_\cT$ is a closed immersion, the K\"ahler metric $\omega_\cT$ is complete on $\cC_\cT$.
	%\end{remark}
	\begin{definition}\label{M-compact}
	    In the situation of Definition \ref{curve-assoc-tree}, given positive constants $\theta$, $\tau$ and $\alpha_v$ for $v\in V$ (with $\tau\le\frac12$ and $\alpha_v\le\theta\le\frac16$ for each $v$) we define the compact subset $\cM_\cT(\theta,\tau,\bm{\alpha})\subset\bA_\cT$ by the inequalities
	    \begin{align}
	        \label{M-def1}|z_{v,e}|&\le\theta \\
	        \label{M-def2}\alpha_v\le|\rho_{v,e}|&\le 2\theta\\
	        \label{M-def3}|\rho_{v,e}| + |\rho_{v,e'}| &\le\tau|z_{v,e} - z_{v,e'}|
	    \end{align}
	    for each edge $e$ and vertex $v = e^-$ (and each edge $e'\ne e$ with $v = e'^-$) and the inequalities
	    \begin{align}
	        \label{M-def4}|\gamma_e|\le\tau
	    \end{align}
	    for each full edge $e\in E_\text{int}$. We are using the notation $\bm\alpha = \{\alpha_v\}_{v\in V}$.
	\end{definition}
	\begin{lemma}\label{M-compact-incl}
	    In the situation of Definition \ref{M-compact}, we have $\cM_\cT(\theta,\tau,\bm{\alpha})\subset\cM_\cT$.
	\end{lemma}
	\begin{proof}
	    For each $u\in V$ and $e\in u^+$, we have $|\rho_{u,e}|\ge\alpha_u>0$. Next, take $u\in V$ and $v,e$ (resp. $v',e'$) as in the first product in \eqref{function-def-M}. Now, on $\cM_\cT(\theta,\tau,\alpha)$, using \eqref{marked-point-across-components} and $\theta\tau\le\frac14$ we get
	    \begin{align}
	        |z^u_{v,e} - z_{u,e_1}|\le\tau|\rho_{u,e_1}|\cdot\sum_{k=2}^\infty(2\theta\tau)^{k-2}\theta = \frac{\theta\tau}{1-2\theta\tau}|\rho_{u,e_1}|\le\textstyle\frac12|\rho_{u,e_1}|.
	    \end{align}
	    Combining this with a similar estimate for $|z^u_{v',e'} - z_{u,e_1'}|$ we get
	    \begin{align}
	        |z^u_{v,e} - z^u_{v',e'}|&\ge|z_{u,e_1} - z_{u,e_1'}| - |z^u_{v,e} - z_{u,e_1}| - |z^u_{v',e'} - z_{u,e_1'}|\\
	        &\ge(\tau^{-1}-\textstyle\frac12)\cdot(|\rho_{u,e_1}|+|\rho_{u,e_1'}|)\\
	        &\ge \alpha_u\tau^{-1}(2-\tau) > \alpha_u.
	    \end{align}
	    It follows that $F_\cT$ is everywhere non-vanishing on $\cM_\cT(\theta,\tau,\bm\alpha)$.
	\end{proof}
	\begin{definition}
	    In the situation of Definition \ref{M-compact}, we define
	    \begin{align}
	        \cC_\cT(\theta,\tau,\bm\alpha) = \pi_\cT^{-1}(\cM_\cT(\theta,\tau,\bm\alpha))\subset\cC_\cT
	    \end{align}
	    and we continue to denote the (restriction of the) projection as $\pi_\cT$.
	\end{definition}
	\begin{lemma}\label{compact-curve-prop}
	    Fix a point $p^*= (\{\gamma_e^*\},\{(z_{v,e}^*,\rho_{v,e}^*)\})\in\cM_\cT(\theta,\tau,\bm\alpha)$. For each $e\in E$ and $v\in\partial(e)$, define the circle $\widehat\Gamma_{v,e}\subset\bC$ as follows. If $v = e^+$, then set $\widehat\Gamma_{v,e} = S^1(1)$. If $v = e^-$, then set $\widehat\Gamma_{v,e} = S^1(z^*_{v,e},|\rho^*_{v,e}|)$. Below, we regard $\bC$ as a subset of $\bP^1$ using the inclusion $z\mapsto[z:1]$.
	    \begin{enumerate}[\normalfont(i)]
	        \item For fixed $v\in V$, the circles $\widehat\Gamma_{v,e}$ enclose pairwise disjoint discs. More precisely, the circles $\{\widehat\Gamma_{v,e}\}_{e\in v^+}$ all lie in $B^2(\frac12)$ and for any $1\le c<\tau^{-1}$, the discs enclosed by the $c$-dilations of these circles about their centres are pairwise disjoint.
	        
	        \item The map $\pi_{v_0}:\cC_{p^*,\cT}\to\bP^1$ is an isomorphism over the set of $[1:z]$ with $z\in B^2(1)$.
	        
	        \item Let $e_0\ne f\in E_\text{ext}$ be a half edge and set $w = f^-$. Then, the projection $\pi_w:
	        \cC_{p^*,\cT}\to\bP^1$ is a local isomorphism over $B^2(z^*_{w,f},|\rho^*_{w,f}|)$.
	        
	        \item Consider $e\in E_\text{int}$. Write $u = e^-$ and $v = e^+$ and let $\cT_u$, with vertex set $u
	        \in V_u$, and $\cT_v$, with vertex set $v\in V_v$, be the trees obtained by splitting $\cT$ along $e$ in the sense of Definition \ref{splitting-def}. Let $p^*_u$, resp. $p^*_v$, be the projections of $p^*$ to $\cM_{\cT_u}(\theta,\tau,\alpha)$, resp. $\cM_{\cT_v}(\theta,\tau,\alpha)$. By {\normalfont(ii)} and {\normalfont(iii)}, the subsets
	        \begin{align}
	            \label{u-disk}
	            D_u(p^*) = \pi_u^{-1}(B^2(z^*_{u,e},|\rho^*_{u,e}|))&\subset\cC_{p^*_u,\cT_u}\\
	            \label{v-disk}
	            D_v(p^*) = \pi_v^{-1}(\bP^1\setminus\normalfont\text{int}\,B^2(1))&\subset\cC_{p^*_v,\cT_v}
	        \end{align}
	        are conformal to $B^2(1)$ via the coordinates $z_u := \frac1{\rho^*_{u,e}} (\frac{x_u}{y_u}-z^*_{u,e})$ and $z'_v: =\frac{y_v}{x_v}$. The curve $\cC_{p^*,\cT}$ is then obtained by deleting $\{|z_u|<|\gamma^*_e|\}\subset D_u(p^*)$, $\{|z'_v|<|\gamma^*_e|\}\subset D_v(p^*)$ and performing the identification
	        \begin{align}\label{local-plumbing}
	            z_uz'_v = \gamma^*_e.
	        \end{align}
	        
	    \end{enumerate}
	\end{lemma}
	\begin{proof}
	    See Appendix \ref{curve-family-proofs} for the proof.
	\end{proof}
	\begin{definition}\label{T-decomp}
	    Let $p^*\in\cM_\cT(\theta,\tau,\bm\alpha)$ be any point. Using the notation of Remark \ref{compact-curve-prop}, for each $e\in E$ and $v\in\partial(e)$, we define the subsets
	    \begin{align}
	        \Gamma_{v,e}(p^*) = \pi_v^{-1}(\widehat\Gamma_{v,e})\subset\cC_{p^*,\cT}
	    \end{align}
	    which can be seen to be circles by Remark \ref{compact-curve-prop}(ii)--(iii). Moreover, by induction on the number of vertices of $\cT$ and Remark \ref{compact-curve-prop}(iv), we see that these circles are pairwise disjoint. For each $v\in V$, define the region
	    \begin{align}
	        R_v(p^*) = \pi_v^{-1}\left(B^2(1)\setminus\textstyle\bigsqcup_{e\in v^+}\text{int}\,B^2(z^*_{v,e},|\rho^*_{v,e}|)\right)\subset\cC_{p^*,\cT}.
	    \end{align}
	    For each $e\in E_\text{int}$ with $u = e^-$ and $v = e^+$, define
	    \begin{align}
	        R_e(p^*) = (D_u(p^*)\times D_v(p^*))\cap\cC_{p^*,\cT}
	    \end{align}
	    where we have used the notation in \eqref{u-disk}--\eqref{v-disk} from Remark \ref{compact-curve-prop}(iv). Finally for $e\in E_\text{ext}$, define
	    \begin{align}
	        R_e(p^*) &= \pi_{v_0}^{-1}(\bP^1\setminus\text{int}\,B^2(1)) \subset\cC_{p^*,\cT}
	    \end{align}
	    when $e = e_0$ and 
	    \begin{align}
	        R_e(p^*) &= \pi_v^{-1}\left(B^2(z^*_{u,e},|\rho^*_{u,e}|)\right) \subset\cC_{p^*,\cT}
	    \end{align}
	    when $e\ne e_0$ (and $u = e^-$). Notice that the regions $\{R_v(p^*)\}_{v\in V}$ are pairwise disjoint, as are the regions $\{R_e(p^*)\}_{e\in E}$. For $v\in V$ and $e\in E$, the regions $R_v(p^*)$ and $R_e(p^*)$ meet precisely when $v\in\partial(e)$ and in this case we have the identity
	    \begin{align}
	        \Gamma_{v,e}(p^*) = R_v(p^*)\cap R_e(p^*).
	    \end{align}
	    The data of the indexed collection $\{\Gamma_{v,e}(p^*)\}$, $\{R_v(p^*)\}$ and $\{R_e(p^*)\}$ is called the \emph{\textbf{thick-thin decomposition}} of $\cC_{p^*,\cT}$ modeled on $\cT$ or, more briefly, the \emph{\textbf{$\cT$-decomposition}} of $\cC_{p^*,\cT}$. We refer to the regions $R_v(p^*)$ for $v\in V$ as the \emph{\textbf{thick}} regions. Similarly, for $e\in E_\text{int}$, resp. $e\in E_\text{ext}$, we refer to the regions $R_e(p^*)$ as \emph{\textbf{necks}}, resp. \emph{\textbf{ends}}.
	\end{definition}
	\begin{definition}\label{T-decomp-metric}
	    Continuing in the situation of Definition \ref{T-decomp}, let $v\in V$ be any vertex and $e\in E$ be any edge. Define the \emph{\textbf{thick distance}} on the region $R_v(p^*)\subset\cC_{p^*,\cT}$ by
	    \begin{align}
	        \dist_{v} = \pi_{v,e_v}^*\dist
	    \end{align}
	    where $e_v$ is the unique edge with $e_v^+ = v$ and $\pi_{v,e_v}:\bP_\cT\to\bP^1$ is the associated projection (as in Definition 
	    \ref{metric-embedding}). Similarly, define the \emph{\textbf{end}} (resp. \emph{\textbf{neck}}) \emph{\textbf{distance}} on the region $R_e(p^*)\subset\cC_{p^*,\cT}$ by
	    \begin{align}
	        \dist_{e} = \max_{u\in\partial(e)} \pi^*_{u,e}\dist
	    \end{align}
	    when $e\in E_\text{ext}$ (resp. $E_\text{int}$). Given a map $f:\cC_{p^*,\cT}\to X$, define the \emph{\textbf{$\cT$-Lipschitz constant}} of $f$, denoted $\text{Lip}_\cT(f)$, to be the indexed family $\{\text{Lip}(f|_{R_w(p)},\dist_{w})\}_{w\in V\cup E}$ of Lipschitz constants of the restrictions of $f$ to the thick/end/neck regions of the $\cT$-decomposition of $\cC_{p^*,\cT}$ with respect to the corresponding thick/end/neck metrics.
	\end{definition}
	\begin{lemma}\label{annulus-dist}
	    In the setting of Definition \ref{T-decomp-metric}, let $e\in\normalfont E_\text{int}$ and let $R_e(p^*)\subset\cC_{p^*,\cT}$ be the corresponding neck region, embedded into $\bP_e = \bP^1\times\bP^1$ via the map $\pi_{u,e}\times\pi_{v,e}$ where $u = e^-$ and $v = e^+$. Given two points $q,q'\in R_e(p^*)$, there exists a piecewise $C^1$ path from $q$ to $q'$ in $R_e(p^*)$ with length $\le 16\pi\cdot\dist_{e}(q,q')$ in $\bP_e$.
	\end{lemma}
	\begin{proof}
	    In view of Lemma \ref{disc-round-to-flat}, the result follows from Lemma \ref{annulus-flat-dist}.
	\end{proof}
	%%%%%%%%%%%%%%%%%%%%%%%%%%%%%%%%%%%%
	\subsection{Holomorphic maps}\label{map-family}
	\begin{definition}
	    Given a stable rooted tree $\cT$, endow the space $\bP_\cT\times X$ with the distance function given by the maximum of the distance function on $\bP_\cT$ and $\lambda_0^{-1}$ times the distance function on $X$. Here, $\lambda_0$ is the constant appearing in Lemma \ref{ell-boot}. Then, the space $\cK(\bP_\cT\times X)$ of compact subsets of $\bP_\cT\times X$, endowed with the Hausdorff distance metric $d_H$ (see Definition \ref{hdorff-distance-defined}) is a complete metric space by Lemma \ref{hdorff-metric-complete}. Let $\cM_\cT(X,J)$ be the set consisting of pairs $(p,f:\cC_{p,\cT}\to X)$, where $p\in\cM_\cT$ and $f$ is a $J$-holomorphic map. Define the \emph{\textbf{graph embedding}}
	    \begin{align}
	        \Gamma_\cT:\cM_\cT(X,J)\to\cM_\cT\times\cK(\bP_\cT\times X)
	    \end{align}
	    by mapping $(p,f:\cC_{p,\cT}\to X)$ to the pair $(p,\Gamma_f)$, where $\Gamma_f$ is the graph of the map $f$, embedded as a compact subset of $\bP_\cT\times X$ via the embedding $\iota_{p,\cT}:\cC_{p,\cT}\to\bP_\cT$ (see Definition \ref{proj-embedding}). The target of $\Gamma_\cT$ has a natural distance function given by taking the maximum of the distance functions of $\cM_\cT$ and $\cK(\bP_\cT\times X)$. Endow $\cM_\cT(X,J)$ with the $\Gamma_\cT$ pullback of this metric.
	\end{definition}
	\begin{definition}\label{compact-space-of-maps}
	    Let $\theta,\tau,\bm\alpha$ be as in Definition \ref{M-compact}. Given positive constants $\eta, \Lambda_w$ (for each $w\in V\cup E$), a compact subset $K\subset X$ and a marking $F\subset E_\text{ext}$ (in the sense of Definition \ref{mark-def}), we define the subset
	    \begin{align}\label{inclusion-of-compact}
	        \cM_{\cT,F}(X,J;\theta,\tau,\bm\alpha;K,\eta,\bm\Lambda)\subset\cM_\cT(X,J)
	    \end{align}
	    to consist of those pairs $(p,f:\cC_{p,\cT}\to X)$ which satisfy conditions \eqref{compact-part-of-moduli}--\eqref{terminal-energy-lb} below.
	    \begin{align}
	        \label{compact-part-of-moduli}
	        p&\in\cM_\cT(\theta,\tau,\bm\alpha)\\
	        \label{compact-image}
	        f(\cC_{p,\cT})&\subset K \\
	        \label{lip-bound}
	        \text{Lip}_\cT(f)&\le\bm\Lambda\lambda_0\\
	        \label{terminal-energy-lb}
	        E(f,R_e(p))&\ge(\eta\lambda_0)^2
	    \end{align}
	    with the \eqref{terminal-energy-lb} holding for each $e\in E_\text{ext}\setminus F$. The notation $\text{Lip}_\cT$ in \eqref{lip-bound} is introduced in Definition \ref{T-decomp-metric} and the \eqref{lip-bound} is to be interpreted as a collection of inequalities between the corresponding entries of $\text{Lip}_\cT(f)$ and $\bm\Lambda = \{\Lambda_w\}_{w\in V\cup E}$.
	\end{definition}
	\begin{remark}
	    Since $\cT$ is stable and we have energy lower bounds \eqref{terminal-energy-lb}, it follows that for any pair $(p,f:\cC_{p,\cT}\to X)$ lying in $\cM_{\cT,F}(X,J;\theta,\tau,\bm\alpha;K,\eta,\bm\Lambda)$, the tuple $(\cC_{p,\cT},\{\sigma_\cT(e,p)\}_{e\in F},f)$ is a stable $J$-holomorphic map. Here, the marked points $\sigma_\cT(e,p)$ are as introduced in Lemma \ref{prop-family-curve}(v).
	\end{remark}
	\begin{lemma}
	    The space
	    \begin{align}
	        \cM_{\cT,F}(X,J;\theta,\tau,\bm\alpha;K,\eta,\bm\Lambda)
	    \end{align}
	    endowed with the pullback metric from the inclusion \eqref{inclusion-of-compact} is compact. 
	\end{lemma}
    \begin{proof}
        Via $\Gamma_\cT$, the space $\cM_{\cT,F}(X,J;\theta,\tau,\bm\alpha;K,\eta,\bm\Lambda)$ is a subset of $\cM_\cT(\theta,\tau,\bm\alpha)\times\cK(\bP_\cT\times K)$. The latter is totally bounded by Corollary \ref{K-net}. Thus, the former is also totally bounded and we only need to check that it is closed in the metric space $\cM_\cT\times\cK(\bP_\cT\times X)$, which is complete by Lemma \ref{hdorff-metric-complete}. For this, take a sequence $(p_i,f_i)$ of points in $\cM_{\cT,F}(X,J;\theta,\tau,\bm\alpha;K,\eta,\bm\Lambda)$ such that we have a point $(p,L)\in\cM_\cT\times\cK(\bP_\cT\times X)$ with $\dist(p_i,p)\to 0$ and $d_H(\Gamma_{f_i},L)\to 0$ as $i\to\infty$. We will construct a $J$-holomorphic map $f:\cC_{p,\cT}\to X$ satisfying \eqref{compact-part-of-moduli}--\eqref{terminal-energy-lb} such that $L = \Gamma_f$.
        \begin{enumerate}[(a)]
            \item We have $L\subset\cC_{p,\cT}\times K$ and its projection to the first coordinate has image $\cC_{p,\cT}$ (here we have identified $\cC_{p,\cT}$ via $\iota_{p,\cT}$ with a subset of $\bP_\cT$).
            \begin{proof}
                Since $p_i\to p$ and $\Gamma_{f_i}\subset\cC_{p_i,\cT}\times X$, it follows that the projection of $L$ to $\bP_\cT$ is a subset of $\cC_{p,\cT}$. Now, if $q\in \cC_{p,\cT}$ is a point not lying in the image of $L$, then we can find a pre-compact open neighborhood $q\in U\subset\cC_\cT$ such that $L\cap(\bar U\times X) = \varnothing$. But then, this implies that $\Gamma_{f_i}$ is disjoint from $U\times X$ as $i\to\infty$, a contradiction. The fact that the projection of $L$ lies in $K$ follows from the fact that \eqref{compact-image} holds for each $f_i$.
            \end{proof}
            \item There exist a continuous map $f:\cC_{p,\cT}\to X$ such that $L = \Gamma_f$ and \eqref{compact-part-of-moduli}--\eqref{lip-bound} hold.
            \begin{proof}
                We just need to produce a map $f:\cC_{p,\cT}\to X$ such that $L = \Gamma_f$ and \eqref{lip-bound} holds. Indeed, by (a), we will then get \eqref{compact-image} and by $p_i\in\cM_\cT(\theta,\tau,\bm\alpha)$ for all $i$, we get \eqref{compact-part-of-moduli}. Let $q,q'\in\cC_{p,\cT}$ be two points, lying in the same thick/neck/end region of the $\cT$-decomposition of $\cC_{p,\cT}$, such that there are points $x,x'\in X$ such that $(q,x),(q',x')\in L$. By convergence in the Hausdorff distance metric, we can find points $q_i,q_i'\in\cC_{p_i,\cT}$ such that $q_i\to q$, $q_i'\to q'$, $f_i(q_i)\to x$ and $f_i(q_i')\to x'$ as $i\to\infty$. We can also find a sequence of points $s_i,s_i'\in\cC_{p_i,\cT}$, which lie on the same thick/neck/end region as $q_i,q_i'$ respectively (as well as the same thick/neck/end region as $q,q'$ respectively), and satisfy $s_i\to q$ and $s_i'\to q'$ as $i\to\infty$. By the uniform Lipschitz bound \eqref{lip-bound} on $f_i$, we find
                \begin{align}
                    \lim_i\dist(f_i(s_i),f_i(q_i))\le \sup_w\Lambda_w\lambda_0\cdot\lim_i\dist(s_i,q_i) = 0\\
                    \lim_i\dist(f_i(s_i'),f_i(q_i'))\le \sup_w\Lambda_w\lambda_0\cdot\lim_i\dist(s_i',q_i') = 0
                \end{align}
                and thus, we can replace $q_i,q_i'$ by $s_i,s_i'$ without loss of generality. Let $\hat w\in V\cup E$ be a vertex or edge for which $s_i,s_i'\in R_{\hat w}(p_i)$ and $q,q'\in R_{\hat w}(p)$. From \eqref{lip-bound} for each $f_i$, we get the estimate
                \begin{align}
                    \dist(x,x') = \lim_i\dist(f_i(s_i),f_i(s_i'))\le\Lambda_{\hat w}\lambda_0\cdot\lim_i\dist_{\hat w}(s_i,s_i') = \Lambda_{\hat w}\lambda_0\cdot\dist_{\hat w}(q,q').
                \end{align}
                The last estimate has the following implications. First, if $q = q'$, then we get $x = x'$. Thus, $L$ is the graph of a set function $f:\cC_{p,\cT}\to X$, which is necessarily continuous since $L$ is compact. Moreover, the same estimate shows that $f$ also satisfies \eqref{lip-bound}.
            \end{proof}
        \end{enumerate}
        Now, we are just left to show that $f$ is $J$-holomorphic and that it satisfies the energy bound \eqref{terminal-energy-lb}. Fix any smooth point $q\in\cC_{p,\cT}$. Find a holomorphic coordinate neighborhood $q\in U\subset\cC_{p,\cT}$ and consider a holomorphic family of coordinate neighborhoods $\varphi:U\times V\to\cC_{p,\cT}$, where $p\in V\subset\cM_\cT$. Convergence of $\Gamma_{f_i}$ to $\Gamma_f$ in Hausdorff distance now shows that the maps $g_i = f_i\circ\varphi(\cdot,p_i)$ converge to $f|_U$ in $C^0_\text{loc}(U,X)$. Moreover, the $g_i$ all enjoy a uniform $C^1_\text{loc}$ bound on $U$ as a consequence of \eqref{lip-bound} and the $C^0_\text{loc}$ convergence. Now, using elliptic bootstrapping (Lemma \ref{ell-boot}), we conclude that $g_i\to f|_U$ in $C^\infty_\text{loc}$. In particular, $f$ is $J$-holomorphic near all of its smooth points. Applying this argument at the points $q = \sigma_\cT(p,e)$, for $e\in E_\text{ext}\setminus F$, with $U = \text{int }R_e(p)$ immediately gives \eqref{terminal-energy-lb}. Finally, since the $f_i$ all enjoy a uniform energy bound (by virtue of the Lipschitz bound \eqref{lip-bound}), it follows that $f$ also has finite energy and thus, by the removable singularity theorem (\cite[Theorem 4.1.2]{McSa}), $f$ is $J$-holomorphic near the singular points of $\cC_{p,\cT}$ also.
    \end{proof}
	%%%%%%%%%%%%%%%%%%%%%%%%%%%%%%%%%%%%%%%%%%%
	\subsection{Effective compactness}\label{eff-net-size}
	We will now estimate the size of the minimal $\delta$-net (see Definition \ref{net-defined}) in $\cM_{\cT,F}(X,J;\theta,\tau,\bm\alpha;K,\eta,\bm\Lambda)$ for any given $\delta\le 1$, explicitly in terms of $\delta$. See Appendix \ref{hdorff-appx} for the relevant notation and background on nets.
	\begin{lemma}\label{sphere-net}
	    Given $0<\gamma<\pi$, we have
	    \begin{align}
	        \nu(\bP^1,\gamma)\le\frac 2{1-\cos(\frac\gamma2)}\le\frac{8\pi}{\gamma^2}.
	    \end{align}
	\end{lemma}
	\begin{proof}
	    Note that the area of a disk of radius $\delta\le\frac\pi2$ in $\bP^1$ is given by $2\pi(1-\cos\delta)$, while the total area of $\bP^1$ is $4\pi$. Let $\{q_1,\ldots,q_N\}$ be a maximal collection of points such that $\dist(q_i,q_j)\ge\gamma$ for all $1\le i<j\le N$. From maximality, it follows that $\{q_1,\ldots,q_N\}$ is a $\gamma$-net. Now, since the open $\frac\gamma2$-balls centred at the $q_i$ are pairwise disjoint, we have
	    \begin{align}
	        \nu(\bP^1,\gamma)\le N\le\frac{4\pi}{2\pi(1-\cos(\textstyle\frac\gamma2))} = \frac 2{1-\cos(\frac\gamma2)}
	    \end{align}
	    as desired. We conclude using $1-\cos\tau\ge\frac{\tau^2}\pi$ which holds for all $0\le\tau\le\frac\pi2$.
	\end{proof}
	%\begin{corollary}
	    %Let $m\ge 1$ and endow $(\bP^1)^m$ with the distance function $\dist(q,q') = \max_{1\le i\le m}\dist(q_i,q_i')$. Then, for any $0<\gamma<\pi$, we have $\nu((\bP^1)^m,\gamma)\le(\frac{8\pi}{\gamma^2})^m$.
	%\end{corollary}
	\begin{theorem}\label{holo-map-tot-bound}
	    Given $0<\delta\le 1$, space $Z = \cM_{\cT,F}(X,J;\theta,\tau,\bm\alpha;K,\eta,\bm\Lambda)$ is covered by finitely subsets, each of diameter $<4\delta$, with the number of sets in the cover being at most
	    \begin{align}\label{eff-net-size-estimate}
	        \left(\frac4{\delta^2}\right)^{\mu-1}\cdot\left(1 + \sigma \delta^{-2k}\cdot\nu(K,\lambda_0)\right)^{(8\pi)^\mu(\frac{\Lambda}{\delta})^{2\mu}(\mu+1)}    
	    \end{align}
	    where $\mu := \sum_{v\in\cT}\deg(v)$, $\Lambda := \sup_w\Lambda_w$, $\nu(K,\lambda_0)$ is the size of the smallest $\lambda_0$-net in the compact metric space $K$, $\sigma = \sigma(X_0)$ is a geometric constant depending on $X$ and $2k = \dim X$. In particular, $\nu(Z,4\delta)$ is bounded above by the number \eqref{eff-net-size-estimate}.
	\end{theorem}
	\begin{proof}
	    Choose a $\delta$-net $A$ for $\cM_\cT(\theta,\tau,\bm\alpha)$ and a $(\lambda_0\delta)$-net for $K$. For each $w\in V\cup E$, choose a $(\delta/\Lambda_w)$-net $B_w$ for $\bP_\cT$ and set $B = \bigsqcup_{w}B_w$. Choose $0<\gamma<\delta$ such that $A,B_w,C$ are remain nets even if $\delta$ is replaced by $\gamma$ in the previous sentences. Define $D = C\sqcup\{*\}$. Now, given any pair $(a,h)\in A\times D^B$ consisting of an element $a\in A$ and a set map $h:B\to D$, define the subset
	    \begin{align}
	        Z_{(a,h)}\subset Z
	    \end{align}
	    as follows. It consists of elements $(p,f:\cC_{p,\cT}\to X)$ in $Z$ such that $\dist(t,a)\le\gamma$ and such that for each $w\in V\cup E$ and each $b\in B_w\subset B$, we have
	    \begin{enumerate}[(i)]
	        \item if $h(b) = *$, then we have $\dist(b,R_w(p)) > \gamma/\Lambda_w$ in $\bP_\cT$ and,
	        \item if $h(b) = c\in C$, then there exists a point $b_p\in R_w(p)$ such that $\dist(b,b_p)\le\gamma/\Lambda_w$ and $\dist(f(b_p),c)\le\gamma$.
	    \end{enumerate}
	    By the choices of $A,B_w,C$, it follows immediately that the sets $Z_{(a,h)}$, ranging over all $(a,h)\in A\times D^B$ cover $Z$. It remains to estimate the diameter of each $Z_{(a,h)}$ and the sizes of $A,B,C$ explicitly in terms of $\delta$ to finish the proof.
	    For any $(p,f:\cC_{p,\cT}\to X)\in Z$ and any $w\in V\cup E$, it follows from \eqref{lip-bound} and the definition of $\dist_{w}$ (see Definition \ref{T-decomp-metric}) that $f|_{R_w(p)}$ is has Lipschitz constant $\le\Lambda_w$ with respect the distance function induced on $R_w(p)$ via its inclusion in $\bP_\cT$. Therefore, by the same argument as in the proof of Lemma \ref{map-space-net}, we find that each of the sets $Z_{(a,h)}$ has diameter $\le 4\gamma<4\delta$.\\\\
	    With $\mu = \sum_{v\in V}\deg(v)$ and $\Lambda = \max_w\Lambda_w$, Lemma \ref{sphere-net} implies that we can take $|B_w|\le (8\pi)^\mu(\frac\Lambda\delta)^{2\mu}$ for each $w\in V\cup E$. Moreover, note that $|V| + |E| = \mu + 1$. Next, note that since $\max\{2\theta, \tau\}<1$, the set $\cM_\cT(\theta,\tau,\bm\alpha)\subset\bA_\cT = \bC^{\mu-1}$ is contained within the set of points where each (complex) coordinate has modulus $\le 1$. Thus, we can choose $A$ such that $|A|\le (\frac4{\delta^2})^{\mu-1}$. Observe also that we have
	    \begin{align}
	        |C|\le \sigma\delta^{-2k}\cdot\nu(K,\lambda_0)
	    \end{align}
	    with $2k = \dim X$ and some explicit constant $\sigma = \sigma(X_0)$, depending on the geometry of $X$. Putting everything together, we find that we have covered $Z$ using finitely many sets each of diameter $<4\delta$, with the number of sets in this cover being $\le|A|\cdot(1 + |C|)^{|B|}$.
	\end{proof}
	\begin{remark}
	    Notice that in the estimate of Lemma \ref{holo-map-tot-bound}, the values of $\eta,\theta,\tau,\bm\alpha$, the individual values of the entries of $\bm\Lambda$, and the set $F$ play no role. It is possible to use the specific values of these to make a get a more precise estimate (i.e., one which is $o$\eqref{eff-net-size-estimate} as $\delta\to 0$).
	\end{remark}
	\end{section}
	%%%%%%%%%%%%%%%%%%%%%%%%%%%%%%%%%%%%%%%%%%%%%%%%
	%%%%%%%%%%%%%%%%% COVERING %%%%%%%%%%%%%%%%%%%%%
	%%%%%%%%%%%%%%%%%%%%%%%%%%%%%%%%%%%%%%%%%%%%%%%%
	\begin{section}{Covering the moduli space by compact sets}\label{covering}
	The aim of this section is to show how, given $\ell\ge 0$ and $A>0$, to cover the moduli space
	\begin{align}
	    \Mbar_{0,\ell}(X,J;K)^{\le A}    
	\end{align}
	with (the images of the energy $\le A$ loci of) finitely many spaces of the form
	\begin{align}
	    \cM_{\cT,F}(X,J;\theta,\tau,\bm\alpha;K,\eta,\bm\Lambda)
	\end{align}
	discussed in \textsection\ref{map-family}. For this whole cover, we will be able to arrange the following.
	\begin{enumerate}[(i)]
	    \item The number of elements in $F$ is exactly $\ell$,
	    \item $\theta,\tau,\eta$  are constants independent of $A,\ell$ and,
	    \item $\Lambda,\bm\alpha$ and the (maximum possible) number of edges of $\cT$ are given by explicit functions of $A,\ell$.
	\end{enumerate}
	\begin{remark}[Choice of constants]\label{cov-const}
	    Before stating the main result of this section, we choose some positive constants -- $\epsilon$ (an absolute constant) and $\lambda$ (depending on $\epsilon$ and the bounds on the geometry of $X$) -- which will play an important role in the estimates to follow. The choices of these constants are only subject to the following constraints.
	    \begin{align}\label{const-1}
	        8\epsilon&\le1\\
	        \label{const-2}
	        9\lambda\sqrt C&\le l\epsilon^2(1-\epsilon)\\
	        \label{const-3}
	        \pi\lambda^2&\le q\epsilon^2
	    \end{align}
	    with equality in (at least) one of \eqref{const-2} and \eqref{const-3}.
	    In \eqref{const-2}, the constants $C$ and $l$ are the ones appearing in Lemmas \ref{mvi} and \ref{exp-decay} (respectively). Similarly, in \eqref{const-3}, $q$ is the constant appearing in Lemma \ref{mvi}.
	\end{remark}
	\begin{theorem}\label{main-covering-theorem}
	    There exist a geometric constant $M = M(X_2)\ge 1$ such that for any positive $\epsilon,\lambda$ satisfying \eqref{const-1}--\eqref{const-3}, the following statement holds. Given an integer $\ell\ge 0$, a number $A\ge 0$ and a genus $0$ stable $J$-holomorphic map $(\Sigma,x_1,\ldots,x_\ell,f)$ into a compact subset $K\subset X$ and with energy $E(f)\le A$, we can find a stable rooted tree $\cT$ with
	    \begin{align}
	        |\normalfont E_\text{ext}|\le\lfloor A/\lambda^2\rfloor + \ell + 1,
	    \end{align}
	    a subset $F\subset\normalfont E_\text{ext}$, in bijection with $\{x_1,\ldots,x_\ell\}$, and numbers $\theta,\tau,\bm\alpha,\eta,\bm\Lambda$, depending explicitly on $\epsilon,\lambda,M, \cT$, such that $(\Sigma,x_1,\ldots,x_\ell,f)$ is isomorphic as a stable map to a point of the metric space
	    \begin{align}
	        \cM_{\cT,F}(X,J;\theta,\tau,\bm\alpha;K,\eta,\bm\Lambda).
	    \end{align}
	    introduced in Definition \ref{compact-space-of-maps}. In fact, we can take
	    \begin{align}
	        \theta &= \epsilon \\
	        \tau &= 4\epsilon \\
	        \alpha_v &= (4\epsilon^3)^{\deg(v)}\\
	        \eta &= \textstyle\frac{1}{3\sqrt C}\cdot\textstyle\frac{\lambda}{\lambda_0}\\
	        \Lambda_v &= \textstyle\frac{9\pi\sqrt C}{\epsilon^2}\cdot\alpha_v^{-1}\\
	        \Lambda_e &= \textstyle\frac{M}{\epsilon^2}
	    \end{align}
	    for vertices $v\in V$ and edges $e\in E$. Here, $C$ and $\lambda_0$ are the constants from Lemmas \ref{mvi} and \ref{ell-boot} respectively.
	\end{theorem}
	\begin{proof}
	    Consider each irreducible component of $\Sigma$ as the domain of a smooth stable map (with additional marked points given by the nodes) and apply Proposition \ref{final-grad-bound} to each component inductively. Notice that end regions (corresponding to same node) of two irreducible components will together form a neck region for the whole stable map $(\Sigma,x_1,\ldots,x_\ell,f)$. Taking $M = 2M'$ will suffice to establish the Lipschitz bound on such (degenerate) neck regions, where $M'$ is the constant provided by Proposition \ref{final-grad-bound}.
	\end{proof}
	%%%%%%%%%%%%%%%%%%%%%%%%%%%%%%%%%%%%%%%%%%
	\subsection{Gradient bound modulo bubbling}
	We will first prove Proposition \ref{c-mod-b}, which is a quantitative analogue of \cite[Theorem 4.6.1]{McSa}.
	\begin{definition}\label{bubble-config-def}
	    Let $\epsilon$ be the constant introduced in Remark \ref{cov-const}. A \emph{\textbf{bubble configuration}} is simply a pair $(T,\rho)$ such that $T\subset\bC$ is a finite set (of \emph{\textbf{bubble points}}) and $\rho:T\to[0,\infty)$ is a function (called the \emph{\textbf{radius function}}). The bubble configuration $(T,\rho)$ is said to be \emph{\textbf{of type $\epsilon$}} if it has the following properties.
	    \begin{enumerate}[(i)]
	        \item We have $\sup_{z\in T}|z|\le\epsilon$.
	        \item We have $\rho(z)\le4\epsilon$ for all $z\in T$.
	        \item For all $x,y\in T$ such that $x\ne y$, we have the estimate
	        \begin{align}
	            \rho(x) + \rho(y)\le \textstyle\frac{\epsilon^2}4|x - y|
	        \end{align}
	    \end{enumerate}
	    If we additionally have $0\in T$ and equality in (i), we say that the bubble configuration $(T,\rho)$ is \emph{\textbf{standard}}. Note that if $|T|\ge 2$, then condition (ii) is superfluous in view of (i) and (iii).
	\end{definition}
	\begin{proposition}\label{c-mod-b}
	    Let $A\ge 0$ be a real number, $\ell\ge 0$ be an integer and let $(\Sigma,x_1,\ldots,x_\ell,f)$ be a stable $J$-holomorphic map with $\Sigma\simeq\bP^1$ and $E(f)\le A$. There is then an isomorphism $\Phi:\bP^1\to\Sigma$ and a standard bubble configuration $(S,r)$ of type $\epsilon$ with the following properties.
	    \begin{enumerate}[\normalfont(i)]
	        \item If $\ell\ge 1$, then we have $\Phi([1:0]) = x_1$ and we define $S_0\subset\bC$ to be the set of $z$ such that $\Phi([z:1]) = x_i$ for some $1<i\le\ell$. If $\ell = 0$, we define $S_0 = \varnothing$. In both cases, we have $S_0\subset S$ and $r^{-1}(0) = S_0$.
	        \item Defining $f_\Phi(z)=(f\circ\Phi)([z:1])$, we have the energy identity
	        \begin{align}\label{energy-cluster}
	            E(f_\Phi,B^2(x,r(x))) = \lambda^2
	        \end{align}
	        for each $x\in S\setminus S_0$. In particular, $|S|\le |S_0| + \lfloor A/\lambda^2\rfloor$. When $\ell = 0$, we also have the energy bound
	        \begin{align}\label{e-l-b}
	            E(f_\Phi,\bC\setminus B^2(r'))\ge\textstyle\frac19C^{-1}(\lambda/r')^2
	        \end{align}
	        for all $r'\ge 1$, where $C$, resp. $\lambda$, are the constants from Lemma \ref{mvi}, resp. Remark \ref{cov-const}.
	        \item For all $z\in\bC$, we have the gradient bound
	        \begin{align}\label{c-mod-b-grad-bound}
	            |df_\Phi(z)|\le\frac{8\sqrt C}{\epsilon^2}\cdot\frac{\lambda}{|z - S|}
	        \end{align}
	        where the constants $C,\lambda$ are as in {\normalfont\text{(ii)}} above and we are using the standard metric on $\bC$.
	    \end{enumerate}
	\end{proposition}
	\noindent The proof is carried out in six steps that follow.\\\\
	\emph{\underline{Step 1}} (Intial choice of gauge) There exists an isomorphism $\varphi:\bP^1\to\Sigma$ such that the $J$-holomorphic map $g:\bC\to X$ given by $g(z) = (f\circ\varphi)([z:1])$ satisfies the following properties.
	\begin{enumerate}[(i)]
	    \item We have the estimate $\|dg\|_\infty\le\lambda$, with equality if $\ell = 0$.
	    \item If $\ell = 0$, then for all $0\le r\le 1$, we have the lower bound $E(g,B^2(r))\ge C^{-1}(\lambda r)^2$. If $\ell\ge 1$, then we have $\varphi([0:1]) = x_1$ and $\varphi^{-1}(x_j) = [1:z_j]$ for some $|z_j|\le\epsilon$ for each $1<j\le\ell$.
	\end{enumerate}
	Here, we are using the standard metric on $\bC$ and the constants $C,\lambda$ are as before.
	\begin{proof}
	    If $\ell\ge 1$, start with any isomorphism $\psi:\bP^1\to\Sigma$ for which $\psi([0:1]) = x_1$. Now, taking the isomorphism $\varphi([z:w]) = \psi([z:Tw])$ for a positive number $T\gg 1$ gives the desired result.\\
	    \indent If $\ell = 0$, then start with any isomorphism $\varphi:\bP^1\to\Sigma$ for which the maximum $M$ of $|d(f\circ\psi)|_{\omega_{\bP^1}}$ occurs at $[0:1]$. Define $\varphi([z:w]) = \psi([\lambda z:2Mw])$ and notice that we now have $\lambda = |dg(0)| = \|dg\|_\infty$. Using \eqref{const-3} we get $\lambda^2\le q\le Cq$ and thus, by Lemma \ref{mvi}, we get the asserted energy lower bound.
	\end{proof}
	\noindent Define $h:\bC\to X$ by $h(z) = (f\circ\varphi)([1:z])$. Choose $z_j\in\bC$ such that $\varphi([1:z_j]) = x_j$ for $1<j\le\ell$. In view of Step 1(i), we have the estimate $|dh(z)|\le\lambda/|z|^2$ for all $z\in\bC$.\\\\
	\noindent\emph{\underline{Step 2}} (Energy threshold radii) We skip this step if $f$ is constant. Assuming $f$ is non-constant, for each $w\in\bC$, define $\fr(w)$ to be the smallest number $\hat\rho>0$ such that $E(h,B^2(w,\hat\rho)) = \lambda^2$. Then, the function $\fr:\bC\to(0,\infty)$ satisfies the following properties.
	\begin{enumerate}[(i)]
	    \item We have $\inf_{w\in\bC}\fr(w) > 0$.
	    \item $\fr$ is $1$-Lipschitz on $\bC$. In particular, it is continuous.
	    \item At any $w\in\bC$, we have $\fr(w)\cdot|dh(w)|\le\lambda\sqrt C$.
	\end{enumerate}
	\begin{proof}
	    Since $h$ is non-constant, we conclude that $E(h)>q\ge\lambda^2$, where the second inequality comes from \eqref{const-3} and \eqref{const-1}. Indeed, if not, then we could apply Lemma \ref{mvi} to large balls with arbitrary centres to conclude $dh \equiv 0$. This shows that $\fr$ is well-defined and takes values in $(0,\infty)$.
	    The bound (iii) now follows immediately by applying Lemma \ref{mvi} to $B^2(w,\fr(w))$.\\
	    \indent To show (ii), we argue by contradiction. Suppose $x,y\in\bC$ are such that $\fr(x) < \fr(y) - |x - y|$. Then, $B^2(x,\fr(x))$ must be properly contained in $B^2(y,\fr(y))$ and both of them carry $\lambda^2$ of the energy $E(h)$, which is a contradiction.\\
	    \indent Finally, to prove (i), start by observing that $|dh(w)|\to 0$ as $|w|\to\infty$, since $h$ extends smoothly at $\infty$. Now consider any given $\gamma>0$. Let $K_\gamma$ be the (compact) set of points which lie at a distance $\le\gamma^{-1}$ from the set $\{z\in\bC\;|\;\sqrt\pi|dh(z)|\ge\gamma\lambda\}$. Then, for any $w\in\bC\setminus K_\gamma$, we must have $|\fr(w)|\ge\gamma^{-1}$. Indeed, if not, then we would have the estimate
	    \begin{align}
	        E(h,B^2(w,\fr(w)))\le\pi\fr(w)^2\left(\frac{\gamma\lambda}{\sqrt\pi}\right)^2 < \lambda^2
	    \end{align}
	    which contradicts the defintion of $\fr$. Since $\gamma$ was arbitrary, this proves that $\fr(w)\to\infty$ as $|w|\to\infty$. This now implies the (weaker) statement that $\inf_{w\in\bC}\fr(w)>0$.
	\end{proof}
	\noindent\emph{\underline{Step 3}} (Selection of bubble points) Define $T_0 = \{z_j\;|\;1<j\le\ell\}$. Then, there exists a bubble configuration $(T,\rho)$ satisfying conditions (i), (iii) of Definition \ref{bubble-config-def} with the following additional properties.
	\begin{enumerate}[(a)]
	    \item We have $T_0\subset T$. Moreover, $\rho\equiv 0$ on $T_0$ and, if $h$ is non-constant, then $\rho\equiv\fr$ on $T\setminus T_0$.
	    \item For any $z\in\bC$ such that $|dh(z)|\ge\lambda/\epsilon^2$, there exists a point $x\in T$ such that the inequalities $\rho(x)\le\fr(z)$ and $\frac{\epsilon^2}{4}|z - x|<\fr(z)  +\rho(x)$ hold.
	    \item We have the bound $|T|\le|T_0| + \lfloor A/\lambda^2\rfloor$.
	\end{enumerate}
	\begin{proof}
        If $h$ is constant, then we can simply take $T = T_0$ and $\rho\equiv 0$. Now, assume that $h$ is non-constant. Consider the set of all bubble configurations $(T,\rho)$ satisfying conditions (i), (iii) of Definition \ref{bubble-config-def} and having property (a) above and the property (P) below.
        \begin{enumerate}[(P)]
            \item Let $B$ be the set of $z\in\bC$ such that $|dh(z)|\ge\lambda/\epsilon^2$ and let $B'$ be the set of points $z\in B$ such that $\frac{\epsilon^2}{4}|z - x|\ge \fr(z) + \rho(x)$ for all $x\in T$. Then, for all $z\in B\setminus B'$, there exists $x\in T$ such that $\rho(x)\le\fr(z)$ and $\frac{\epsilon^2}{4}|z - x|<\fr(z) + \rho(x)$.
        \end{enumerate}
        Obviously, the set of such $(T,\rho)$ is partially ordered (by inclusion of $T$ and restriction of $\rho$) and non-empty (since we can take $T = T_0$ and $\rho\equiv 0$). Moreover, since $\frac{\epsilon^2}4<1$ from \eqref{const-1}, it follows from condition (iii) of Definition \ref{bubble-config-def} that $\{B^2(x,\fr(x))\}_{x\in T\setminus T_0}$ are pairwise disjoint and therefore,
        \begin{align}
            |T\setminus T_0|\le\lfloor E(h)/\lambda^2\rfloor\le\lfloor A/\lambda^2\rfloor.
        \end{align}
        Thus, we may take a maximal such $(T,\rho)$. It satisfies (a), (c), (P) by construction and we are left to check that it also satisfies (b). Suppose it does not. Then, the (compact) set $B'$ from (P) must be non-empty. Let $x\in B'$ be a point where $\fr|_{B'}$ attains its minimum. Define $T' = T\cup\{x\}$ and $\rho'$ on $T'$ by $\rho'|_{T} = \rho$ and $\rho'(x) = \fr(x)$. By Step 1(i) and $|dh(x)|\ge\lambda/\epsilon^2$, we get $|x|\le\epsilon$. It is now easy to verify that $(T',\rho')$ is also a bubble configuration satisfying conditions (i), (iii) of Definition \ref{bubble-config-def} and enjoying the properties (a) and (P). This contradicts the maximality of $(T,\rho)$.
	\end{proof}
	\noindent\emph{\underline{Step 4}} (Gradient bound) $T$ is non-empty. Moreover, for any $w\in\bC$, we have the gradient estimate
	\begin{align}\label{prelim-grad-bound}
	    |dh(w)|\le\frac{8\sqrt C}{\epsilon^2}\cdot\frac{\lambda}{|w - T|}.
	\end{align}
	\begin{proof}
	    Suppose $T = \varnothing$. Then, $T_0 = \varnothing$ and $h$ must be non-constant by stability. In view of (b) above, we deduce that we must have $\|dh\|_{\infty}\le\lambda/\epsilon^2$. Using this to estimate the energy of $h$ on $B^2(\epsilon)$ and using Step 1(i) to estimate the energy of $h$ on $\bC\setminus B^2(\epsilon)$, we see that $E(h)\le\pi(\lambda\epsilon^{-1})^2\le q$, where the last inequality is just \eqref{const-3}. Now, arguing as in the beginning of Step 2, we get a contradiction to the non-constancy of $h$.\\
	    \indent Now, suppose $w\in\bC$ is any point. If $|dh(w)|\ge\lambda/\epsilon^2$, then $|w|\le\epsilon$ we can find a point $x\in T$ such that
	    \begin{align}
	        \textstyle\frac{\epsilon^2}4|w-x|&<\fr(w) + \rho(x)\\
	        \rho(x)&\le\fr(w)
	    \end{align}
	    and thus, by Step 2(iii), we get the desired gradient bound. On the other hand, if $|dh(w)|\le\lambda/\epsilon^2$ and $|w|\le 1$, then we can use $|w-T|\le 1+\epsilon$ to verify the gradient estimate. Finally, if $|w|\ge 1$, then Step 1(i) gives $|dh(w)|\le\lambda/|w|^2\le\lambda/|w|\le\lambda(1+\epsilon)/|w-T|$, which again gives the desired gradient estimate. Here, we have used $\sup_{x\in T}|x|\le\epsilon\le\epsilon|w|$ to estimate $\frac{|w-T|}{|w|}$ by $1+\epsilon$.
	\end{proof}
	\noindent\emph{\underline{Step 5}} (At least two bubble points) In fact, $|T|\ge 2$.
	\begin{proof}
	    If $h$ is constant, then we must have $|T_0|\ge 2$ by stability. Thus, in view of Step 4, it is enough to rule out the case when $T_0 = \varnothing$, $|T| = 1$ and $h$ is non-constant. Indeed, let $a\in T$ be the unique element in this case. Consider the isomorphism $\nu_a:\bR\times S^1\to\bC\setminus\{a\}$ given by
	    \begin{align}
	        (s,t)\mapsto a + e^{-(s+it)}
	    \end{align}
	    and note that, by \eqref{const-2} and \eqref{prelim-grad-bound}, we have $\|d(h\circ\nu_a)\|_\infty\le l$, where $l$ is as in Lemma \ref{exp-decay}. Applying \eqref{action-diff} from Lemma \ref{exp-decay}, we deduce that $\int f^*\omega = \int (h\circ\nu_a)^*\omega = 0$. Since $f$ is $J$-holomorphic and $J$ is $\omega$-tame, this shows that $f$ must be constant, a contradiction.
	\end{proof}
	\noindent\emph{\underline{Step 6}} (Completion of proof)
	\begin{proof}
	    We are now ready to define $\Phi$, $S$ and $r$ and verify their properties. First, define the quantity
	    \begin{align}
	        \kappa = \sup_{x,y\in T}\epsilon^{-1}|x - y|.
	    \end{align}
	    We have $0<\kappa\le 2$ in view of Steps 3 and 5. Choose $x^*,y^*\in T$ such that $|x^*-y^*| = \epsilon\kappa$. Define the isomorphism $\Phi:\bP^1\to\Sigma$ by the formula
	    \begin{align}
	        [z:w]\mapsto \varphi([w:x^*w+\kappa z])
	    \end{align}
	    where $\varphi$ is as in Step 1. Define $S = \kappa^{-1}(T - x^*)$ and $r:S\to(0,\infty)$ by $z\mapsto \kappa^{-1}\rho(x^* + \kappa z)$.\\ \indent Clearly, $(S,r)$ is a bubble configuration satisfying (i), (iii) of Definition \ref{bubble-config-def}. Moreover, by considering the images of $x^*,y^*$ in $S$, we see that $0\in S$ and equality occurs in (i). Thus, $(S,r)$ is a standard bubble configuration of type $\epsilon$. Assertion (i) of Proposition \ref{c-mod-b} as well as the energy identity \eqref{energy-cluster} are clear from the corresponding statements for $(T,\rho)$. The gradient bound \eqref{c-mod-b-grad-bound} for $f_\Phi$ follows immediately from \eqref{prelim-grad-bound}. Thus, we are only left to verify \eqref{e-l-b}. For any $r'\ge 1$, we have
	    \begin{align}
	        E(h,\bC\setminus B^2(x^*,r'\kappa))\ge E(h,\bC\setminus B^2(|x^*| + r'\kappa)).
	    \end{align}
	    Combining this with $|x^*| + r'\kappa\le\epsilon + 2r'\le 3r'$ and the energy lower bound of Step 1(ii) gives \eqref{e-l-b}.
	\end{proof}
	%%%%%%%%%%%%%%%%%%%%%%%%%%%%%%%%%%%%%%%%%%
	\subsection{Combinatorics of bubble configurations}
	To proceed further in the proof of Theorem \ref{main-covering-theorem}, we will need to investigate the recursive structure of bubble configurations. For this we will use the next two elementary lemmas.
	\begin{lemma}\label{cluster-selection}
	    Suppose $(Z,d)$ is a finite metric space and assume $(a_i)_{i\ge 0}$ is a sequence of positive real numbers such that $a_{i+1}\le\frac12 a_i$ for all $i\ge 0$. Then, given any point $s\in Z$, there exists a subset $Z'\subset Z$ containing $s$ such that the following properties hold.
	    \begin{enumerate}[\normalfont(i)]
	        \item For $x\ne y$ in $Z'$, we have $|x - y| > a_{|Z'|-1}$.
	        \item There is a retraction $R:Z\to Z'$ such that $d(R(x),x)\le a_{|Z'|}$ for all $x\in Z$.
	    \end{enumerate}
	    Moreover, in this case, the map $R$ is uniquely determined by $Z'$.
	\end{lemma}
	\begin{proof}
	    Note that the subset $\{s\}$ satisfies (i). Pick a subset $s\in Z'\subset Z$ satisfying (i), which is maximal with respect to inclusion. If $Z'$ violates (ii), then there exists a point $z\in Z\setminus Z'$ such that $d(z,Z')>a_{|Z'|}$. But then $Z'\cup\{z\}$ also satisfies (i), and this is a contradiction to the maximality of $Z'$. Finally, the uniqueness of $R$ follows from the triangle inequality and the fact that $2a_k\le a_{k-1}$ for all $k\ge 1$.
	\end{proof}
	\noindent In the following lemma, the constant $\epsilon$ is again as in Remark \ref{cov-const}.
	\begin{lemma}\label{reduce}
	    Given a standard bubble configuration $(T,\rho)$ of type $\epsilon$, apply Lemma \ref{cluster-selection} to the finite metric space $T$, sequence $a_i = (4\epsilon^3)^i$ for $i\ge 0$ and $s = 0$ to obtain a subset $T'\subset T$ and a retraction $R:T\to T'$. We then have $k = |T'|\ge 2$. Define the function $\rho':T'\to(0,\infty)$ by the formula
	    \begin{align}\label{rho'-def}
	        \rho'(x) = \textstyle\frac14\epsilon^{-1}\cdot\max\{\epsilon^{-1}(4\epsilon^3)^k,\rho(x)\}.
	    \end{align}
	    Then, the balls $\{B^2(x,\rho'(x))\}_{x\in T'}$ are pairwise disjoint and contained in $B^2(2\epsilon)$. In fact, we have the following stronger estimate for all $x\ne y$ in $T'$.
	    \begin{align}\label{2-sep}
	        2\epsilon|x - y| \ge\rho'(x) + \rho'(y).
	    \end{align}
	    Next, for each $z\in T$, we have
	    \begin{align}
	        \label{cluster-diam}
	        |z-R(z)|&\le4\epsilon^2\cdot \rho'(R(z))\\
	        \label{rescaling-rad-v-cutoff-rad}
	        \rho(z)&\le4\epsilon\cdot\rho'(R(z))
	    \end{align}
	    and, if $z\ne R(z)$, then we additionally have $\rho'(R(z)) = \frac14\epsilon^{-2}(4\epsilon^3)^{k}$. Finally, for each $x\in T'$ and any $w\in B^2(x,\rho'(x))$, the minimum $|w - T|$ is attained by a point in $R^{-1}(x) = T\cap B^2(x,\rho'(x))$.
	\end{lemma}
	\begin{proof}
	    Since $(T,\rho)$ is standard, we have $\sup_{z\in T}|z| = \epsilon > 4\epsilon^3$ and thus, $k = |T'|\ge 2$. Now, fix $x\in T'$. Choose any $z\in T\setminus\{x\}$ and note that
	    \begin{align}\label{rho'-bound}
	        \rho'(x) = \textstyle\frac14\epsilon^{-1}\cdot\max\{\epsilon^{-1}(4\epsilon^3)^k,\rho(x)\}\le
	        \max\{\epsilon,\textstyle\frac{1}{16}\epsilon\cdot |x - z|\}\le\max\{\epsilon,\frac{\epsilon^2}{8}\} = \epsilon
	    \end{align}
	    and thus, $|x| + \rho'(x)\le 2\epsilon$. Next, to prove \eqref{2-sep}, consider $x\ne y\in T'$. If $\rho'(x) = \rho'(y) = \frac14\epsilon^{-2}(4\epsilon^3)^k$, then we have $\rho'(x) + \rho'(y) = \frac12\epsilon^{-2}(4\epsilon^3)^k = \frac12\cdot4\epsilon\cdot(4\epsilon^3)^{k-1}<2\epsilon|x - y|$. If not, assume (without loss of generality) that $4\epsilon\rho'(x) = \rho(x)\ge\rho(y)$ and note that
	    \begin{align}
	        \rho'(x) + \rho'(y) \le 2\rho'(x) =  \textstyle\frac12\epsilon^{-1}\rho(x)\le\frac12\epsilon^{-1}\cdot\frac{\epsilon^2}4|x  -y| < 2\epsilon|x - y|
	    \end{align}
	    as desired. Now, take any $z\in T$ such that $R(z) = x$. Using the definition of $R$ and $\rho'$, we immediately get
	    \begin{align}
	        |z - x|\le (4\epsilon^3)^k = 4\epsilon^2\cdot\textstyle\frac14\epsilon^{-2}(4\epsilon^3)^k\le4\epsilon^2\rho'(x)
	    \end{align}
	    and that $\rho(x)\le 4\epsilon\rho'(x)$. If $z\ne x$, then we must have
	    \begin{align}
	        \rho(z) + \rho(x) \le \textstyle\frac{\epsilon^2}4|x - z|\le\frac{\epsilon^2}4(4\epsilon^3)^k<\epsilon^{-1}(4\epsilon^3)^k\le4\epsilon\rho'(x)
	    \end{align}
	    which simultaneously proves equation \eqref{rescaling-rad-v-cutoff-rad} and the fact that $\rho'(x) = \frac14\epsilon^{-2}(4\epsilon^3)^k$. To prove the last assertion, take any $w\in B^2(x,\rho'(x))$. In view of \eqref{cluster-diam}, it is clear that $R^{-1}(x) = T\cap B^2(x,\rho'(x))$. Consider any $u\in T$ and put $v = R(u)\in T'$. If $v\ne x$, then we have
	    \begin{align}
	        |w - u| - |w - x| &\ge |x - v| - |u - v| - 2|w - x|\\
	        &\ge |x - v| - |u - v| - 2\rho'(x)\\
	        & \ge (1-4\epsilon)|x - v| - |u - v| 
	    \end{align}
	    where we used \eqref{2-sep} to remove the quantity $4\epsilon|x - v| - 2\rho'(x)\ge 0$. Since $\frac{|u - v|}{|x - v|}< 4\epsilon^3 < 1 - 4\epsilon$, we get $|w - u| > |w - x|$. Thus, the minimum $|w - T|$ is attained on $R^{-1}(x)$.
	\end{proof}
	\begin{remark}\label{reduction}
	    Continuing in the situation of Lemma \ref{reduce}, we can decompose $T' = T'_\text{ext}\sqcup T'_\text{int}$, according to whether or not $x\in T'$ satisfies $R^{-1}(x) = \{x\}$. Now, consider any $x\in T'_\text{int}$. Then, we define
	    \begin{align}\label{gluing-parameter}
	        0<\gamma(x) = \sup_{z\in R^{-1}(x)}\frac{|z - x|}{\epsilon\rho'(x)} \le 4\epsilon
	    \end{align}
	    and the standard bubble configuration $(T_x,\rho_x)$ of type $\epsilon$ as follows. We set $T_x$ to be the inverse image of $R^{-1}(x) = T\cap B^2(x,\rho'(x))$ under the map $\Phi_x(w) = x + \gamma(x)\rho'(x)w$ and define $\rho_x:T_x\to(0,\infty)$ by
	    \begin{align}
	        \rho_x(w) = \frac{\rho(\Phi_x(w))}{\gamma(x)\rho'(x)}.
	    \end{align}
	    We say that $(T_x,\rho_x)$ is got by \emph{\textbf{reduction}} of $(T,\rho)$ at $x\in T'_\text{int}\subset T'$.
	\end{remark}
	\noindent Next, we show how to associate a stable rooted tree $\cT$ (and a point $p\in\cM_\cT$) to a standard bubble configuration of type $\epsilon$. This will be applied to the bubble configuration provided by Proposition \ref{c-mod-b} to prove Theorem \ref{main-covering-theorem}.
	\begin{definition}\label{tree-assoc-bubble-config}
	    There exists an assignment which maps standard bubble configurations $(T,\rho)$ of type $\epsilon$ to pairs $(\cT,p)$, with $\cT=(V,E,\partial)$ a stable rooted tree and $p\in\bA_\cT$ (see Definition \ref{curve-assoc-tree}) a point, such that the following property holds for any $(T,\rho)$.
	    \begin{enumerate}[\normalfont(A)]
	        \item There exists a subset $T'\subset T$ and retraction $R:T\to T'$ satisfying the conclusion of Lemma \ref{cluster-selection}, applied to the finite metric space $T$ with the sequence $a_i = (4\epsilon^3)^i$ for $i\ge 0$ and $s = 0$, such that $|T'|\ge 2$ and the following assertions are true. Take $\rho'$ as defined in Lemma \ref{reduce}.
	        \begin{enumerate}[\normalfont(a)]
	            \item If $T = T'$, then $\cT$ has a single vertex $v(T,\rho)$, a root edge $e(T,\rho)$ and the remaining edges $\{e_x\}_{x\in T}$ are in bijection with $T$. The point $p\in\cM_\cT$ is defined by declaring $z_{v(T,\rho),e_x}(p) = x$ and $\rho_{v(T,\rho),e_x}(p) = \rho'(x)$ for each $x\in T$.
	            \item If $T\ne T'$, then decompose $T' = \normalfont T'_\text{ext}\sqcup T'_\text{int}$ as in Remark \ref{reduction}. For each $x\in \normalfont T'_\text{int}$, let $(\cT_x,p_x)$ be the pair associated to $(T_x,\rho_x)$. Let $V_x,E_x,\partial_x$ and $e(T_x,\rho_x)$ be the vertex set, edge set, boundary map and root edge of $\cT_x$. Then, $\cT$ is given by
	            \begin{align}
	                V  &= \{v(T,\rho)\}\sqcup\textstyle\bigsqcup_{x\in \normalfont T'_\text{int}} V_x\\
	                E &= \{e(T,\rho)\}\sqcup\textstyle\bigsqcup_{x\in \normalfont T'_\text{int}} E_x\sqcup\{e_x\}_{x\in \normalfont T'_\text{ext}}
	            \end{align}
	            with root edge $e(T,\rho)$, and boundary map $\partial$ defined by
	            setting $\partial(e(T,\rho)) = \partial(e_x) = \{v(T,\rho)\}$ for each $x\in\normalfont T'_\text{ext}$, $\partial(e(T_x,\rho_x)) = \partial_x(e(T_x,\rho_x))\sqcup\{v(T,\rho)\}$ and $\partial = \partial_x$ on $E_x\setminus\{e(T_x,\rho_x)\}$ for each $x\in\normalfont T_\text{int}$. The point $p\in\bA_\cT$, following the notation of Remark \ref{reduction}, is defined by setting
	            \begin{align}
	                \gamma_{e(T_x,\rho_x)}(p) &= \gamma(x)\\ z_{v(T,\rho),e(T_x,\rho_x)}(p) &= x\\ \rho_{v(T,\rho),e(T_x,\rho_x)}(p) &= \rho'(x)\\
	                z_{v(T,\rho),e_y}(p) &= y\\ \rho_{v(T,\rho),e_y}(p) &= \rho'(y)
	            \end{align}
	            for each $x\in\normalfont T'_\text{int}$, $y\in\normalfont T'_\text{ext}$ with the other coordinates $z_{v,e},\rho_{v,e}$ and $\gamma_e$ of $p$ (for $v\in V_x$, $e\in E_x$ with $x\in T'_\text{int}$) given by the corresponding coordinates of $p_x$.
	        \end{enumerate}
	    \end{enumerate}
	    Fix an assignment satisfying (A) once and for all. We will refer to $(\cT,p)$ as the tree and moduli point \emph{\textbf{associated}} to the standard bubble configuration $(T,\rho)$ of type $\epsilon$.
	\end{definition}
	\begin{lemma}\label{tree-assoc-bubble-config-prop}
	    Use the notation of Definition \ref{tree-assoc-bubble-config} and let $(T,\rho)$ be a standard bubble configuration of type $\epsilon$ and let $(\cT,p)$ be the associated tree and moduli point. Define $\theta = \epsilon$, $\tau = 4\epsilon$ and $\bm\alpha = \{\alpha_v\}_{v\in V}$ by $\alpha_v \equiv (4\epsilon^3)^{\deg(v)}$. Then, we have the following properties.
	    \begin{enumerate}[\normalfont(i)]
	        \item $p\in\cM_\cT(\theta,\tau,\bm\alpha)$, in the sense of Definition \ref{M-compact}.
	        \item For each half-edge $e\ne e(T,\rho)$ of $\cT$, let $v_e\in\partial(e)$ be its unique endpoint. Then, the map $e\mapsto z_{v_e,e}^{v(T,\rho)}(p)$, in the notation of Definition \ref{curve-assoc-tree}, defines a bijection $\normalfont E_\text{ext}\setminus\{e(T,\rho)\}\to T$.
	        \item All the $\gamma$-coordinates of $p$ are $\ne 0$. In particular, the projection $\pi_{v(T,\rho)}:\cC_{p,\cT}\to\bP^1$ is an isomorphism.
	    \end{enumerate}
	\end{lemma}
	\begin{proof}
	    The proof is by induction on $|T|$, using property (A) from Definition \ref{tree-assoc-bubble-config}. Start with $(T,\rho)$ and find a retraction $R$ of $T$ onto a subset $T'\subset T$ such that (a), (b) from Definition \ref{tree-assoc-bubble-config} hold. Assume that (i), (ii) hold for all standard bubble configurations of type $\epsilon$ which have strictly fewer bubble points.
	    \begin{enumerate}[(i)]
	        \item If $T\ne T'$, then property (i) of Definition \ref{bubble-config-def}, the definition of $\rho'(\cdot)$ in Lemma \ref{reduce} and the estimates \eqref{2-sep} and \eqref{gluing-parameter} allow us to verify \eqref{M-def1}--\eqref{M-def4}, showing inductively that $p\in\cM_\cT(\theta,\tau,\bm\alpha)$. The case when $T = T'$ is even simpler since \eqref{M-def4} is vacuous in this case and \eqref{M-def1}--\eqref{M-def3} are verified as before.
	        \item When $T = T'$, the assertion is clear. When $T\ne T'$, the assertion follows using
	        \begin{align}
	            R^{-1}(T'_\text{int}) &= \textstyle\bigsqcup_{x\in T'_\text{int}} \Phi_x(T_x)\\
	            z^{v(T,\rho)}_{v_e,e}(p) &= \Phi_x(z^{v(T_x,\rho_x)}_{v_e,e}(p_x))
	        \end{align}
	        for each $x\in T'_\text{int}$ and the fact that (ii) holds for each $(T_x,\rho_x)$. Here, $\Phi_x$ is as in Remark \ref{reduction}.
	        \item The fact that each $\gamma$-coordinate of $p$ is nonzero is obvious from induction and \eqref{gluing-parameter}. Thus, for any vertex $v$ of $\cT$, the projection $\pi_v:\cC_{p,\cT}\to\bP^1$ is an isomorphism by Lemma \ref{prop-family-curve}(i). In particular, we can apply this to $v = v(T,\rho)$.
	    \end{enumerate}
	\end{proof}
	%%%%%%%%%%%%%%%%%%%%%%%%%%%%%%%%%%%%%%%%%%
	\subsection{Choice of gauge and gradient bound}
	\begin{definition}\label{final-gauge}
	    Let $A\ge 0$ be a real number, $\ell\ge 0$ be an integer and let $(\Sigma,x_1,\ldots,x_\ell,f)$ be a stable $J$-holomorphic map with $\Sigma\simeq\bP^1$ and $E(f)\le A$. Let $\Phi:\bP^1\to\Sigma$ be the isomorphism and $(S,r)$ be the standard bubble configuration of type $\epsilon$ given by Proposition \ref{c-mod-b}. Let $(\cT,p)$ be the tree and moduli point associated to $(S,r)$ by Definition \ref{tree-assoc-bubble-config}. Define the map $\Psi:\cC_{p,\cT}\to\Sigma$ by
	    \begin{align}
	        \Psi = \Phi\circ\pi_{v(T,\rho)}
	    \end{align}
	    which is an isomorphism using Lemma \ref{tree-assoc-bubble-config-prop}. If $\ell = 0$, we define $F\subset E_\text{ext}$ be the empty set. If $\ell\ge 1$, we define $F\subset E_\text{ext}$ to consist of the root edge $e(T,\rho)$ and the edges $e\in E_\text{ext}\setminus\{e(T,\rho)\}$ for which $r(z^{v(T,\rho)}_{v_e,e}(p)) = 0$, where we are using the notation of Lemma \ref{tree-assoc-bubble-config-prop}(ii). It is easy to see that the marked points $\{\sigma_\cT(e,p)\}_{e\in F}$, introduced in Lemma \ref{prop-family-curve}(v), are in bijection with $\{x_1,\ldots,x_\ell\}$. Finally, define $f_\Psi:\cC_{p,\cT}\to X$ as $f_\Psi = f\circ\Psi$.
	\end{definition}
	\begin{proposition}\label{final-grad-bound}
	    In the situation of Definition \ref{final-gauge}, if $K\subset X$ is a compact set containing $f(\Sigma)$, then
	    \begin{align}
	        (p,f_\Psi:\cC_{p,\cT}\to X) \in
	        \cM_{\cT,F}(X,J;\theta,\tau,\bm\alpha;K,\eta,\bm\Lambda)
	    \end{align}
	    where we have
	    \begin{align}
	        \theta &= \epsilon \\
	        \tau &= 4\epsilon \\
	        \alpha_v &= (4\epsilon^3)^{\deg(v)}\\
	        \eta &= \textstyle\frac{1}{3\sqrt C}\cdot\textstyle\frac{\lambda}{\lambda_0}\\
	        \Lambda_v &= \textstyle\frac{9\pi\sqrt C}{\epsilon^2}\cdot\alpha_v^{-1}\\
	        \Lambda_e &= \textstyle\frac{M'}{\epsilon^2}
	    \end{align}
	    for vertices $v\in V$ and edges $e\in E$.
	    Here, $C$ is the constant from Lemma \ref{mvi}, while $\epsilon,\lambda$ are as in Remark \ref{cov-const} and $M' = M'(X_2)\ge 1$ is a constant depending on the bounds on the geometry of $X$. Moreover, \begin{align}
	        |\normalfont E_\text{ext}|\le \lfloor A/\lambda^2\rfloor + \ell + 1.
	    \end{align}
	\end{proposition}
	\begin{proof}
	    From Definition \ref{bubble-config-def}(i), \eqref{rho'-def}, \eqref{2-sep} and \eqref{gluing-parameter}, we immediately see that we can take $\theta = \epsilon$, $\tau = 4\epsilon$ and $\alpha_v = (4\epsilon^3)^{\deg(v)}$ so that \eqref{compact-part-of-moduli} holds. Condition \eqref{compact-image} is obvious from $f(\Sigma)\subset K$. Using Proposition \ref{c-mod-b}(ii), we find that if we take $\eta$ such that $(\eta\lambda_0)^2 = \frac19C^{-1}\lambda^2$, then \eqref{terminal-energy-lb} also holds. We now only need to check that \eqref{lip-bound} holds with the stated choice of $\bm\Lambda$. The main observation is that a bound of the form \eqref{c-mod-b-grad-bound} is invariant under re-scalings and translations of $\bC$ (provided we transform $S$ also accordingly). We will now check \eqref{lip-bound} for $f_\Psi$ by considering the cases of thick/end/neck regions separately.
	    \begin{enumerate}[(a)]
	        \item Let $v\in V$ and consider the neck region $R_v(p)$. We can identify $R_v(p)$ via $\pi_{v,e_v}$ (where $e_v$ is the unique edge with $e_v^+ = v$) with the set
	        \begin{align}
	            R_v'(p) = B^2(1)\setminus\textstyle\bigsqcup_{e\in v^+}\text{int }B^2(z_{v,e},|\rho_{v,e}|) \subset \bC \subset \bP^1
	        \end{align}
	        Write $f_v$ for the restriction of $f_\Psi\circ\pi_{v,e_v}^{-1}$ to this set $R_v'(p)$. We first find a gradient estimate for this map in the metric $\omega_\bC$. In view of Lemma \ref{reduce} and the (re-scaled) estimate \eqref{c-mod-b-grad-bound}, we get
	        \begin{align}
	            |df_v(w)|_{\omega_{\bP^1}}\le |df_v(w)|_{\omega_\bC}\le\frac{8\sqrt C}{\epsilon^2}\cdot\sup_{e\in v^+}\frac{\lambda(1-4\epsilon^2)^{-1}}{|w - z_{v,e}|}\le \frac{9\sqrt C}{\epsilon^2}\cdot\frac{\lambda}{(4\epsilon^3)^{\deg(v)}}
	        \end{align}
	        for all $w\in R'_v(p)$. Now, given any two points $w_1,w_2\in R'_v(p)$, we can join them by a path $\gamma$ in $R'_v(p)$ of $\omega_\bC$-length $\le\pi|w_1 - w_2|$. Indeed, take the straight line path from $w_1$ to $w_2$ and replace each segment which goes through the discs $B^2(z_{v,e},|\rho_{v,e}|)$ by the (smaller) arc along the boundary of this disc. Thus,
	        \begin{align}
	            \dist(f_v(w_1),f_v(w_2))\le\frac{9\sqrt C}{\epsilon^2}\lambda\alpha_v^{-1}\cdot\pi|w_1 - w_2| \le \Lambda_v\cdot\dist_{\omega_{\bP^1}}(w_1,w_2)
	        \end{align}
	        and this proves the desired Lipschitz bound on $R_v'(p)$.
	        \item Let $e\in E_\text{ext}$ and consider the end region $R_e(p)$. First consider the case when $e$ is not the root edge $e(S,r)$. Let $u = e^-$. Using $\pi_{u,e}$, we can identify $R_e(p)$ with $R'_e(p) = B^2(1)\subset\bC\subset\bP^1$. Write $f_e$ for the restriction of $f_\Psi\circ\pi_{u,e}^{-1}$ to this set $R_e'(p)$. The (re-scaled) estimate \eqref{c-mod-b-grad-bound} now tells us that we have
	        \begin{align}\label{terminal-grad}
	            |df_e(w)|_{\omega_\bC}\le\frac{8\sqrt C}{\epsilon^2}\cdot\frac{\lambda}{|w|}
	        \end{align}
	        for all $w\in R_e'(p)$. This estimate blows up as $w\to 0$. But, since we know that $f_e$ extends smoothly to $w = 0$ and that $E(f_e,R_e'(p)) = \lambda^2\le q$, we get
	        \begin{align}
	            |df_e(w)|_{\omega_\bC}\le 2\lambda\sqrt C
	        \end{align}
	        for each $w\in B^2(\frac12)$ by applying Lemma \ref{mvi} to the ball $B^2(w,\frac12)$. Thus, we get $|df_e(w)|_{\omega_{\bP^1}}\le\frac{16\sqrt C}{\epsilon^2}\cdot\lambda$ for all $w\in R_e'(p)$. Now, if $e = e(S,r)$, then let $v = v(S,r)$ be the root vertex, i.e., $v = e^+$. We can use $\pi_{v,e}$ followed by the map $[z:w]\mapsto[w:z]$ to identify $R_e(p)$ with $R_e'(p) = B^2(1)\subset\bC\subset\bP^1$ and let $f_e$ be the map on $R_e'(p)$ defined by $f_\Psi$. It is again immediate to see that \eqref{terminal-grad} again holds. Using cylindrical coordinates on $B^2(1)\setminus\{0\}$ and \eqref{const-2} to see that Lemma \ref{exp-decay} is applicable, we conclude from the identity \eqref{action-diff} and \eqref{weak-iso-peri} that we have
	        \begin{align}
	            E(f_e,R_e'(p))\le \hat c\cdot\left(\frac{8\sqrt C}{\epsilon^2}\cdot2\pi\lambda\right)^2
	        \end{align}
	        where $\hat c = \hat c(X_2)$ is the maximum of the product of the constant $c$ from Lemma \ref{exp-decay} with \begin{align}
	            \sup_{V\ne 0}\frac{g(V,V)}{\omega(V,JV)}
	        \end{align}
	        and 1. Now, arguing as before (using Lemma \ref{mvi}) we get $|df_e(w)|_{\omega_{\bP^1}}\le\frac{32\pi C\sqrt{\hat c}}{\epsilon^2}\cdot\lambda$ for all $w\in R_e'(p)$.
	        \item Let $e\in E_\text{int}$ and consider the neck region $R_e(p)$. Write $u = e^-$ and $v = e^+$. Using $\pi_{u,e}\times\pi_{v,e}$, we can identify $R_e(p)$ with the subset $R_e'(p)\subset\bP^1\times\bP^1$ consisting of pairs $([z:1],[1:w])$ such that $z,w\in B^2(1)$ and $zw = \gamma_e$. Let $f_e$ be the map on $R_e'(p)$ induced by $f_\Psi$.\\\\ 
	        Using $\pi_{u,e_u}$, where $e_u$ is the unique edge with $e_u^+ = u$, we can identify $R_e(p)$ with the region
	        \begin{align}
	           R''_e(p) =  B^2(z_{u,e},|\rho_{u,e}|)\setminus\text{int }B^2(z_{u,e},|\gamma_e\rho_{u,e}|)\subset\bC\subset\bP^1.
	        \end{align}
	        The isomorphism $\varphi_e:R'_e(p)\to R''_e(p)$, which induces the identity on $R_e(p)$, is given by
	        \begin{align}
	            ([z:1],[1:w])\mapsto z_{u,e} + \rho_{u,e}z.
	        \end{align}
	        Note that, in view of \eqref{c-mod-b-grad-bound}, the map $g_e = f_e\circ\varphi_e^{-1}$ satisfies
	        \begin{align}
	            |dg_e(w)|_{\omega_\bC}\le\frac{8\sqrt C}{\epsilon^2}\cdot\frac{\lambda(1-\epsilon)^{-1}}{|w - z_{u,e}|}
	        \end{align}
	        for all $w\in R_e''(p)$. Using cylindrical coordinates
	        \begin{align}
	            \nu:[-R,R]\times S^1\to R_e''(p)
	        \end{align}
	        with $2R = -\log|\gamma_e|$ on this annulus, and \eqref{const-2}, we find that Lemma \ref{exp-decay} is applicable to the map $g_e\circ\nu$. Equations \eqref{action-diff} and \eqref{weak-iso-peri} then give $\sqrt{E(g_e\circ\nu)}\le\frac{16\pi\lambda\sqrt{2\hat cC}}{\epsilon^2(1-\epsilon)}$, where $\hat c$ is an case (b) above. Applying \eqref{exp-decay-estimate}, we easily deduce that we have the estimate
	        \begin{align}\label{cyl-neck-grad-bound}
	            |d(g_e\circ\nu)(s,t)|\le \frac{8\lambda\sqrt C }{\epsilon^2(1-\epsilon)}(e + 2\pi K_1\sqrt{2\hat c})e^{-R}(e^s+e^{-s})
	        \end{align}
	        for \emph{all} $(s,t)\in[-R,R]\times S^1$, where we are using the standard metric $ds^2 + dt^2$ on the domain. By Lemmas \ref{annulus-cyl-flat}, \ref{disc-round-to-flat} and \ref{annulus-dist}, we can now convert \eqref{cyl-neck-grad-bound} into a bound for the Lipschitz constant of $f_e$ on $R_e'(p)$ with respect to $\dist_e$.
	    \end{enumerate}
	    We have succeeded in establishing a Lipschitz bound for $f_\Psi$ on each of the regions in the $\cT$-decomposition of $\cC_{p,\cT}$. The proof is completed once we define the constant $M' = M'(X_2)\ge 1$ in the obvious manner using the maximum of the constants obtained from cases (b) and (c).
	\end{proof}
	\noindent We also note the following observation which will be useful in the next section.
	\begin{proposition}\label{final-Lip-bound}
	    In the situation of Proposition \ref{final-grad-bound}, we have
	    \begin{align}\label{final-Lip-ineq}
	        \dist(f_\Psi(q_1),f_\Psi(q_2))\le 2\Lambda\lambda_0\cdot\dist(q_1,q_2)
	    \end{align}
	    whenever the distance between $q_1,q_2\in\cC_{p,\cT}\subset\bP_\cT$ is $\le 1$, where we set $\Lambda = \sup_{w\in V\cup E}\Lambda_w$.
	\end{proposition}
	\begin{proof}
	    If $q_1,q_2$ are in the same end/neck/thick region, then the statement is obvious from Proposition \ref{final-grad-bound} above. If $\dist(q_1,q_2)\le 1$, then one of the two points must lie in a thick region. Assume, without loss of generality, that $q_1\in R_v(p)$ and $q_2\in R_e(p)$ for some $v\in\partial(e)$. Using the (rescaled) bound \eqref{c-mod-b-grad-bound} and $\dist(q_1,q_2)\le 1$, it is now easy to verify \eqref{final-Lip-ineq}.
	\end{proof}
	\end{section}
	%%%%%%%%%%%%%%%%%%%%%%%%%%%%%%%%%%%%%%%%%%%%%%%%%
	%%%%%%%%%%% PROOF OF MAIN RESULT %%%%%%%%%%%%%%%%
	%%%%%%%%%%%%%%%%%%%%%%%%%%%%%%%%%%%%%%%%%%%%%%%%%
	\begin{section}{Proof of the main result}\label{last-proof}
    In this section, we will prove Theorem \ref{intro-theorem}. %%%%%%%%%%%%%%%%%%%%%%%%%%%%%%%%%%%%%%%%%%%%%%%%%%%%%
    \subsection{Proof of Theorem \ref{intro-theorem}(i)}
    Let $u = [\Sigma,x_1,\ldots,x_\ell,f]\in\Mbar_{0,\ell}(X,J;K)^{\le A}$ be given. Take $\lambda_0$ to be as in Theorem \ref{ell-boot}, $\epsilon = \frac18$ and let $\lambda$ be the largest number satisfying \eqref{const-2}--\eqref{const-3}. Then, apply Theorem \ref{main-covering-theorem} to find a tree $\cT$ with a subset $F\subset E_\text{ext}$ in bijection with $\{x_1,\ldots,x_\ell\}$, constants $\theta,\tau,\bm\alpha,\eta,\bm\Lambda$ and an isomorphism $\Psi:\Sigma\to\cC_{p,\cT}$ of $u$ (as a stable map) with a point $(p,f_\Psi:\cC_{p,\cT}\to X)$ of the space
    \begin{align}
        \cM_{\cT,F}(X,J;\theta,\tau,\bm\alpha;K,\eta,\bm\Lambda).
    \end{align}
    Proposition \ref{final-Lip-bound} shows that $f_\Psi:\cC_{p,\cT}\to X$ has the property that for $q_1,q_2\in\cC_{p,\cT}\subset\bP_\cT$, we have
    \begin{align}\label{final-Lip-2}
        \dist(f_\Psi(q_1),f_\Psi(q_2))\le2|\bm\Lambda|\lambda_0\cdot\dist(q_1,q_2)
    \end{align}
    whenever $\dist(q_1,q_2)\le 1$, where $|\bm\Lambda|$ denotes the maximum of $\Lambda_w$ over $w\in V\cup E$. By inspecting the explicit formula for $\bm\Lambda$ in Theorem \ref{main-covering-theorem} and noting that a non-constant genus $0$ stable $J$-holomorphic map must have energy $\ge q$ (from Lemma \ref{mvi}) it follows that we can replace $2|\bm\Lambda|$ in \eqref{final-Lip-2} by the number $\Lambda$ where
    \begin{align}
        \log\Lambda = c\cdot(\ell + \lfloor A/\lambda^2\rfloor)
    \end{align}
    with $c\ge 9$ being a sufficiently large positive constant (independent of the bounds on the geometry of $X$). Now, consider any pair $(v,e)\in V\times E$ such that $v\in\partial(e)$. By Lemma \ref{tree-edge-bound}, there are
    \begin{align}
        \sum_{v\in V}\deg(v) \le 3\cdot|E_\text{ext}| - 6 \le 3\cdot(\ell + \lfloor A/\lambda^2\rfloor)
    \end{align}
    such pairs. Define $m = c\cdot(\ell + \lfloor A/\lambda^2\rfloor)$. For each pair $(v,e)$ with $v\in\partial(e)$, the map $\pi_{v,e}:\cC_{p,\cT}\subset\bP_\cT\to\bP^1$ is an isomorphism over $S^1\subset\bP^1$, and therefore, we can define the points $a_{v,e},b_{v,e},c_{v,e}\in\cC_{p,\cT}$ by
    \begin{align}
      a_{v,e} &= \pi_{v,e}^{-1}[1:1]\\
      b_{v,e} &= \pi_{v,e}^{-1}[e^{2\pi i/3}:1]\\
      c_{v,e} &= \pi_{v,e}^{-1}[e^{-2\pi i/3}:1].
    \end{align}
    Now, endow $\cC_{p,\cT}$ with $\ell + m$ distinct smooth marked points, where 
    \begin{itemize}
        \item $\ell$ of them consist of $x_1,\ldots,x_\ell$,
        \item $3\cdot\sum_{v\in V}\deg(v)$ of them consist of $\{a_{v,e},b_{v,e},c_{v,e}\}_{v\in\partial(e)}$ and,
        \item the remaining marked points are chosen arbitrarily.
    \end{itemize}
    This induces an $m$-decoration $\varphi:\cC_{p,\cT}\to\cC_s\subset Z_{m+\ell}$, where $s\in\Mbar_{0,\ell + m}$. Moreover, recalling the definition of the metric on $\bP_\cT$, it is immediate that $\varphi^{-1}$ has Lipschitz constant $\le 1$. Thus, the $1$-local Lipschitz constant of $u^\varphi$ is $\le\Lambda$.
    %%%%%%%%%%%%%%%%%%%%%%%%%%%%%%%%%%%%%%%%%%%%%%%%%%%%%
    \subsection{Proof of Theorem \ref{intro-theorem}(ii)}
    Let $B$ be a $(\delta/\Lambda)$-net for $Z_{\ell + m}$ and $C$ be a $(\lambda_0\delta)$-net for $K$. It then follows, as in the proof of Lemma \ref{map-space-net}, that taking $N\ge(1 + |C|)^{|B|}$ will suffice. It therefore remains to estimate $|B|$ and $|C|$. Exactly as in the proof of Theorem \ref{holo-map-tot-bound}, we see that we can take
    \begin{align}
        |B|&\le \left(8\pi\cdot\frac{\Lambda^2}{\delta^2}\right)^{\binom{m+\ell}{3}}\\
        |C|&\le\sigma\delta^{-2k}\cdot\nu(K,\lambda_0)
    \end{align}
    which completes the proof. Here, $\sigma = \sigma(X_0)$ is a constant depending only the bounds on the geometry of $X$ and $\nu(K,\lambda_0)$ is the size of the smallest $\lambda_0$-net in $K$.
	\end{section}
	%%%%%%%%%%%% APPENDICES %%%%%%%%%%%%%%%%%%%%%%%%%
	\appendices
	%%%%%%%%%%%%%%%%%%%%%%%%%%%%%%%%%%%%%%%%%%%%%%%%%%
	%%%%%%%%%%% PROOFS OF A PRIORI ESTIMATES %%%%%%%%%
	%%%%%%%%%%%%%%%%%%%%%%%%%%%%%%%%%%%%%%%%%%%%%%%%%%
	\begin{section}{Proofs of \emph{a priori} estimates}\label{est-appx}
    In this appendix, we prove the \emph{a priori} estimates from Lemmas \ref{ell-boot} and \ref{exp-decay}.
    %%%%%%%%%%%%%%%%%%%%%%%%%%%%%%%%%%%%%%%%%%%%%%
    \subsection{Elliptic bootstrapping (proof of Lemma \ref{ell-boot})}\label{ell-boot-appx}
    Let $\lambda_0 = \lambda_0(X_1)$ be chosen so that for any point $x\in X$, the set of points with $\dist(x,\cdot)\le 2\lambda_0$ lies in a unit coordinate ball (centred at $x$) in which $J,g$ are given by square matrices which (along with their inverses) have $C^k$ norms $\le$ the $C^k$ bounds on the geometry of $X$. Now, let us consider a map $v:B^2(1)\to X$ which is $J$-holomorphic and satisfies $\|dv\|_\infty\le\lambda_0$. Then, using elliptic bootstrapping for the Laplacian in local coordinates (as in the proof of \cite[Lemma B.11.3]{ParVFC}), we obtain constants $c_k = c_k(X_{k+1})$ such that
    \begin{align}\label{ell-boot-local}
        |(\nabla^kv)(0)|\le c_k\lambda_0
    \end{align}
    for each integer $k\ge 2$.\\\\
    Now, consider a $J$-holomorphic map $u:B^2(r)\to X$ (with a bounded first derivative) as in the statement of Lemma \ref{ell-boot}. Define the quantity $\rho>0$ by $\rho\|du\|_\infty = \lambda_0$. First, suppose $\rho\le r$ and define the map $v_\rho:B^2(1)\to X$ by $v_\rho(z) = u(\rho z)$. We then have $\|dv_\rho\|_\infty = \lambda_0$ and thus, by \eqref{ell-boot-local}, we have
    \begin{align}
        \rho^k|(\nabla^k u)(0)| = |(\nabla^k v_\rho)(0)|\le c_k\lambda_0
    \end{align}
    for each $k\ge 2$. On the other hand, if $\rho>r$, we argue as follows. Let $X'$ be the Riemannian manifold which is $X$ endowed with the metric $g' = (\frac \rho r)^2g$ and the corresponding distance function $\dist' = \frac \rho r\dist$. The crucial observation is that the bounds on the geometry of $(X',g',J)$ are $\le$ the bounds on the geometry of $(X,g,J)$. Define $v_r:B^2(1)\to X'$ by $v_r(z) = u(rz)$ and note that $\|dv_r\|'_\infty = \lambda_0$. We then get
    \begin{align}
        r^k|(\nabla^ku)(0)| = |(\nabla^k v_r)(0)| = \textstyle\frac r\rho|(\nabla^k v_r)(0)|'\le c_k\lambda_0(\frac r\rho)
    \end{align}
    for each $k\ge 2$. In particular, the sum of these two bounds is always true (irrespective of the relative sizes of $r$ and $\rho$) and gives the estimate $|(\nabla^k u)(0)|\le \frac{c_k\lambda_0}{\rho}(\frac1r + \frac1\rho)^{k-1}$ as desired.
    %%%%%%%%%%%%%%%%%%%%%%%%%%%%%%%%%%%%%%%%%%%%
    \subsection{Mean value inequality (alternative proof of Lemma \ref{mvi})}\label{alt-mvi}
    Consider any $J$-holomorphic $u:B^2(r)\to X$. Assume that $|du(0)|\ne 0$ since otherwise there is nothing to prove. By Hofer's Lemma (see \cite[Lemma 4.6.4]{McSa}), we can find $0<\rho\le\frac 12r$ and $z_0\in B^2(r-2\rho)$ such that we have $\rho|du(z_0)|\ge\frac12r|du(0)|$ and
    \begin{align}
        \sup_{B^2(z_0,\rho)}|du|\le 2|du(z_0)|
    \end{align}
    Now, define $v:B^2(\rho)\to X$ by $v(z) = u(z + z_0)$. By Lemma \ref{ell-boot} (proved above), we get the estimate
    \begin{align}
        \sup_{B^2(\frac12\rho)}|\nabla^2 v|\le c_2\|dv\|_\infty\left(\frac{2}{\rho} + \frac{\|dv\|_\infty}{\lambda_0}\right) \le 2^2c_2|dv(0)|\left(\frac{1}{\rho} + \frac{|dv(0)|}{\lambda_0}\right)
    \end{align}
    and thus, for $|z|\le\frac12\rho$, we deduce the estimate
    \begin{align}
        \frac{|dv(z)|}{|dv(0)|}\ge 1 - 2^2c_2|z|\left(\frac{1}{\rho} + \frac{|dv(0)|}{\lambda_0}\right).
    \end{align}
    From this, we deduce the simpler lower bound $|dv(z)|\ge\frac12|dv(0)|$ for $|z|\le\rho' := (2^3c_2)^{-1}\left(\frac1\rho + \frac{|dv(0)|}{\lambda_0}\right)^{-1}$. Noting that $2^3\rho'\le\rho$, this now allows us to estimate the energy of $u$ from below as follows.
    \begin{align}\label{mv-alt-est}
        E(u)\ge E(v) =  \frac12\int_{B^2(\rho')}|dv|^2>\frac\pi {2^3}|dv(0)|^2\rho'^2
        \ge\frac{\pi c_2^{-2}}{2^9}\left(\frac2{r|du(0)|} + \frac{1}{\lambda_0}\right)^{-2}
        \ge C^{-1}\min\{r|du(0)|,\lambda_0\}^2
    \end{align}
    where $C$ is defined to make the last inequality true. In the penultimate inequality, we have used the fact that $\rho|dv(0)|\ge\frac12r|du(0)|$. Now, taking $q:=C^{-1}\lambda_0^2$, we get $r|du(0)|\le\sqrt{C\cdot E(u)}$ if $E(u)\le q$, as desired.
    \begin{remark}
        The idea of this proof was suggested by John Pardon. Note that, in this proof, the constants $C,q$ depend on $C^3$ bounds on the geometry of $X$ while the proof of \cite[Lemma 4.3.1(i)]{McSa} needs only $C^2$ bounds.
    \end{remark}
    %%%%%%%%%%%%%%%%%%%%%%%%%%%%%%%%%%%%%%%%%%%%
    \subsection{Gromov--Schwarz and isoperimetric inequality}\label{gs-ip}
    To prove Lemma \ref{exp-decay}, we will need (a slight strengthening of) the Gromov--Schwarz Lemma (\cite[1.2.A \& 1.2.B$'$]{gromov85}) and (a variant of) the symplectic isoperimetric inequality.
    \begin{lemma}\label{gs}
        There exist constants $\varepsilon_0 = \varepsilon_0(X_2)$ and $G$ such that if $u:B^2(r)\to X$ is a $J$-holomorphic map with $\|\dist(u(\cdot),x)\|_\infty\le\varepsilon_0$ for some point $x\in X$, then we have
        \begin{align}
            r^2|du(0)|^2\le \frac{G^2}{r^2}\int_{B^2(r)}|\dist(u(\cdot),x)|^2
        \end{align}
    \end{lemma}
    \begin{proof}
        By the scale invariance of the result, we may assume that $r = 1$. By taking $\varepsilon_0\le\lambda_0$ (from \textsection\ref{ell-boot-appx}), we can assume that $u$ maps into a coordinate ball with $x = 0$ where $\omega,J,g$ are given by square matrices which (along with their inverses) have $C^k$ norms $\le$ the $C^k$ bounds on the geometry of $X$. Further, we can assume that $J(0) = i$. The $J$-holomorphic curve equation for $u$ can be written as
        \begin{align}
            \delbar u = T(u)\cdot\partial u
        \end{align}
        where $T(\cdot)$ is a tensor vanishing at $0$ and satisfying $\|T\|_{C^1}\le G' = G'(X_1)$. Let $0\le \varphi\le 1$ be a smooth cutoff function on $B^2(1)$ which is $\equiv 1$ on $B^2(\frac12)$, $\equiv 0$ outside $B^2(\frac34)$ and satisfies $\|\nabla\varphi\|_\infty\le 5$. Letting $v = \varphi u$, we compute that $v$ satisfies the equation $\delbar v - T(u)\cdot\partial v = \delbar\varphi\cdot u - T(u)\cdot(u\partial\varphi)$ and deduce the estimate
        \begin{align}
            \|\partial v\|_{L^2} = \|\delbar v\|_{L^2}\le G'\varepsilon_0\|\partial v\|_{L^2} + \|u\|_{L^2}(1 + G'\varepsilon_0)\|\nabla\varphi\|_\infty
        \end{align}
        where, in the first step, we have used integration by parts. From this, we get
        \begin{align}
            \sqrt{E(u,B^2(\textstyle\frac12))}\le\|\nabla v\|_{L^2}\le 2\|\partial v\|_{L^2}\le 2\|\nabla\varphi\|_\infty\left(\frac{1+G'\varepsilon_0}{1 - G'\varepsilon_0}\right)\|u\|_{L^2}\le 14\|u\|_{L^2}
        \end{align} 
        with the last two inequalities holding provided we take $G'\varepsilon_0\le\frac16$. Now, if we further require $14^2\pi\varepsilon_0^2\le q$ (where $q$ is the constant from Lemma \ref{mvi}), then we can use Lemma \ref{mvi} to find an explicit constant $G$ for which the inequality $|du(0)|\le G\|u\|_{L^2}$ holds, as desired.
    \end{proof}
    \begin{lemma}\label{iso-peri}
        If $\gamma:S^1\to\bC^n$ is a smooth map, then we have
        \begin{align}
            \left|\int_{S^1}\gamma^*\lambda_{\normalfont\text{std}}\right|\le\frac1{4\pi}\left(\int_{S^1}\left|\dot\gamma(t)\right|dt\right)^2\le\frac12\int_{S^1}\left|\dot\gamma(t)\right|^2dt
        \end{align}
        where $\lambda_{\normalfont\text{std}} = \frac i4\sum_{j=1}^n(z_jd\bar z_j - \bar z_jdz_j)$ is the standard Liouville form on $\bC^n$.
    \end{lemma}
    \begin{proof}
        Denote the three quantities to be compared as (I), (II) and (III). The inequality $(\text{I})\le(\text{III})$ is proved by a simple computation using the Fourier series expansion $\gamma(t) = \sum_m c_me^{-imt}$. The reparametrization invariance of (I) and (II), combined with Cauchy-Schwarz, can then be used to deduce $(\text{I})\le(\text{II})\le(\text{III})$. See \cite[Lemma 4.4.4]{McSa} for the complete details.
    \end{proof}
    %%%%%%%%%%%%%%%%%%%%%%%%%%%%%%%%%%%%%%%%%%%%%%
    \subsection{Long cylinders of small energy (proof of Lemma \ref{exp-decay})}\label{exp-decay-appx}
    The proof expands on ideas appearing in \cite[Lemma 4.7.3]{McSa} and \cite[Lemma 2.3]{HT-obg2}. Let $u:[R_-,R_+]\times S^1\to X$ be a $J$-holomorphic map as in the statement of Lemma \ref{exp-decay}, satisfying $\|du\|_\infty\le l$. Take $l$ to be small enough that the symplectic action is defined for all loops of length $<4\pi l$. Then, Stokes' theorem gives
    \begin{align}\label{action-diff-small}
        \int_{[a,b]\times S^1}u^*\omega = \cA(\gamma_a) - \cA(\gamma_b)
    \end{align}
    which holds for any $R_-\le a\le b\le R+$, provided that $|a-b|$ is sufficiently small. By subdividing $[R_-,R_+]$ into sufficiently small intervals and using a telescoping sum argument, we find that \eqref{action-diff-small} holds even if $|a-b|$ is not small. The estimate $|\cA(\gamma_s)|\le c\cdot\ell(\gamma_s)^2$, for $R_-\le s\le R_+$, is a consequence of Lemma \ref{iso-peri}, the Darboux theorem and the bounds on the geomery of $X$. Next, note that we have the energy estimate
    \begin{align}
        E(u)\le C'\int_{[R_-,R_+]\times S^1} u^*\omega = C'\left(\cA(\gamma_{R_-}) - \cA(\gamma_{R_+})\right)\le 2C'c(2\pi l)^2
    \end{align}
    where $C' = C'(X_0)$ is a positive constant which quantitatively expresses the fact that $\omega$ tames $J$. Now, arguing as in \cite[Lemma 4.7.3]{McSa}, we can find constants $K' = K'(X_2)$ and $\mu = \mu(X_2)$ such that we have
    \begin{align}\label{initial-decay}
        |du(s,t)|\le K'\sqrt{E(u)}\left(e^{\mu(s-R_+)} + e^{\mu(R_--s)}\right)
    \end{align}
    for all $(s,t)\in [R_-+\frac14,R_+-\frac14]\times S^1$. For completeness, we outline this argument in our context. For $0\le r<\frac{R_+-R_-}2$, define the function $f(r) = E(u|_{[R_-+r,R_+-r]\times S^1})$. Differentiate $f$ to get
    \begin{align}
        -4\pi f'(r)&\ge \ell(\gamma_{R_-+r})^2 + \ell(\gamma_{R_+-r})^2\\
        &\ge c^{-1}\left(\cA(\gamma_{R_-+r}) - \cA(\gamma_{R_+-r})\right)\\
        &=c^{-1}\int_{[R_-+r,R_+-r]\times S^1}u^*\omega\\
        &\ge (C'c)^{-1}f(r)
    \end{align}
    for all $0\le r<\frac{R_+-R_-}2$. Integrating this inequality for $f$, combining with Lemma \ref{mvi} and requiring $2C'c(2\pi\ell)^2\le q$ (where $q$ is the constant from Lemma \ref{mvi}) gives the estimate \eqref{initial-decay}. Integrating the estimate for $|du(s,t)|$ shows that $u|_{[R_-+\frac14,R_+-\frac14]\times S^1}$ maps into a coordinate ball in $X$ where we have bounds on the matrices representing $\omega,J,g$ (and their inverses). For simplicity of notation in what follows, let us assume that this is actually true for $u$ on the whole of $[R_-,R_+]\times S^1$. We can adjust the choice of the coordinate system a little more so that we have $J(0) = i$ and
    \begin{align}
        \langle u\rangle = \int_{S^1}u(\textstyle\frac{R_-+R_+}2,\theta)\,d\theta = 0
    \end{align}
    For the rest of the argument, we work exclusively in this coordinate system. Since shrinking $l$ reduces the $L^\infty$ bound for $u$ in this coordinate system, it now suffices to prove the following statement. Lemma \ref{ell-boot} will then give the higher derivative bounds, completing the proof of Lemma \ref{exp-decay}.
    \begin{lemma}\label{exp-decay-local}
        Let $J$ be an almost complex structure on $B^{2n}(1)\subset\bC^n$ with $J(0) = i$. Let \begin{align}
            v:[R_-,R_+]\times S^1\to B^{2n}(1)
        \end{align} 
        be a $J$-holomorphic map satisfying $\langle v\rangle = 0$. There exists a constant $\eta$, depending only on the $C^2$ bounds on $J$, and an absolute constant $H$ with the following significance. If $\|v\|_\infty\le\eta$ and $\|\nabla v\|_{L^2}\le\eta$ then we have
        \begin{align}
            |\nabla v(s,t)|\le H\|\nabla v\|_{L^2}(e^{s-R_+} + e^{R_--s})
        \end{align}
        for all $(s,t)\in[R_-+\frac12,R_+-\frac12]\times S^1$.
    \end{lemma}
    \begin{proof}
        Set $R = \frac12(R_-+R_+)$ and define, for $0\le r< R$, the non-negative function $g$ by
        \begin{align}
            g(r) = \int_{[r+R_-,R_+-r]\times S^1}|\partial_tv|^2\,dt\,ds
        \end{align}
        and note that we have
        \begin{align}
            -{\textstyle\frac12}g'(r) &= {\textstyle\frac12}\int_{S^1} \left(|\partial_tv(R_+-r,t)|^2 + |\partial_tv(r+R_-,t)|^2\right)dt\\
            &\ge\int_{\partial[r+R_-,R_+-r]\times S^1}v^*\lambda_{\text{std}}\\
            &=\int_{[r+R_-,R_+-r]\times S^1}v^*\omega_{\text{std}}\\
            &=\int_{[r+R_-,R_+-r]\times S^1}\langle i\partial_sv,\partial_tv\rangle\,dt\,ds\\
            &\ge g(r)\cdot(1 - c'\|v\|_\infty)
        \end{align}
        where $c' = \|J\|_{C^1}$ and the last estimate follows from $i\partial_sv = (1 - i(J(v) - i))\partial_tv$. Taking $\eta\le\frac12c'^{-1}$, we get $g(r)\le e^{-r}\|\partial_tv\|^2_{L^2}$. Repeating the same computation with $\partial_sv$ instead of $\partial_tv$ and adding the results gives the energy decay estimate
        \begin{align}
            \int_{[r+R_-,R_+-r]\times S^1}|\nabla v|^2\,dt\,ds\le e^{-r}\int_{[R_-,R_+]\times S^1}|\nabla v|^2\,dt\,ds.
        \end{align}
        Using Lemma \ref{mvi} and requiring $\eta^2\le q$ gives a decay estimate for $|\nabla v|$. Integrating this and combining with the assumption that $\langle v\rangle = 0$ gives the estimate
        \begin{align}\label{initial-decay-local}
            |v(s,t)| + |(\nabla v)(s,t)|\le A(e^{\frac12(s-R_+)} + e^{\frac12(R_--s)})\|\nabla v\|_{L^2}
        \end{align}
        for all $(s,t)\in[R_-+\frac14,R_+-\frac14]\times S^1$, where $A$ is an absolute constant. Following arguments modeled on the proof of \cite[Lemma 2.3]{HT-obg2}, we will improve this decay estimate. For convenience, let us assume that \eqref{initial-decay-local} holds on all of $[R_-,R_+]\times S^1$. Fourier expand $v$ into eigenfunctions of the operator $L = i\partial_t$ on $S^1$ to get
        \begin{align}
            v(s,t) = \sum_m v_m(s)e^{-imt}.    
        \end{align}
        Let $\Pi_\pm$ and $\Pi_0$ be the orthogonal projections from $L^2(S^1)$ onto the $\pm$ and $0$ eigenspaces of the operator $L$ and define $v_\alpha = \Pi_\alpha v$ for $\alpha\in\{+,-,0\}$. The Cauchy-Riemann equation for $v$ translates to
        \begin{align}\label{cr-local}
            \partial_sv + Lv = S(v)\cdot Lv
        \end{align}
        where $S(\cdot) = (J(\cdot) - i)i$ and thus, satisfies $|S(v)|\le c'|v|$ pointwise. Define the functions 
        \begin{align}
            f_\pm(s) = \pm\textstyle\frac12\langle Lv_\pm,v_\pm\rangle_{L^2(\{s\}\times S^1)}
        \end{align}
        for $s\in[R_-,R_+]$ and note that applying $\Pi_\pm$ to \eqref{cr-local} and applying $\langle\cdot,Lv_\pm\rangle_{L^2(\{s\}\times S^1)}$ gives
        \begin{align}\label{coupled-1}
            f_+'(s) + \|Lv_+\|^2&\le\hat\varepsilon(s)\left(\|Lv_+\|^2 + \|Lv_-\|^2\right)\\
            \label{coupled-2}
            -f_-'(s) + \|Lv_-\|^2&\le\hat\varepsilon(s)\left(\|Lv_+\|^2 + \|Lv_-\|^2\right)
        \end{align}
        where $\hat\varepsilon(s) = c'A\eta(e^{\frac12(s-R_+)} + e^{\frac12(R_--s)})\le\frac13$ provided $\eta$ is chosen small enough. Taking $\gamma(s) = \frac{\hat\varepsilon(s)}{1-\hat\varepsilon(s)}$, and adding \eqref{coupled-1} to $\gamma(s)$ times \eqref{coupled-2} gives
        \begin{align}
            f'_+ - \gamma f'_- + (1-\gamma)\|Lu_+\|^2\le 0
        \end{align}
        Now, noting that $\|Lu_+\|^2\ge 2f_+$ and that $|\gamma'|\le 2\gamma(1-\gamma)$ for $s\in[R_-,R_+]$, we get
        \begin{align}
            (f_+-\gamma f_-)' + 2(1-\gamma)(f_+-\gamma f_-)\le 0.
        \end{align}
        Observe that $\int_{R_-}^{R_+}\gamma(s)\,ds \le c''$ for a constant $c''$ independent of $R_\pm$ and so, we get
        \begin{align}
            (f_+ - \gamma f_-)(s) &\le e^{2(c''+R_--s)}(f_+-\gamma f_-)(R_-)\\
            (f_- - \gamma f_+)(s) &\le e^{2(c''+s-R_+)}(f_- - \gamma f_+)(R_+)
        \end{align}
        by integrating the previous estimate. Rearranging the last two estimates, we get
        \begin{align}
            f_\pm(s)\le \frac{e^{2c''}}{1-\gamma(s)^2}\left(e^{\pm 2(R_\mp-s)}f_\pm(R_\mp) + \gamma(s)e^{\pm 2(s-R_\pm)}f_\mp(R_\pm)\right).
        \end{align}
        Noting that $2f_\pm(R_\mp)\le\|Lv_\pm\|^2(R_\mp)\le2\pi\cdot(2A\|\nabla v\|_{L^2})^2$ by \eqref{initial-decay-local}, we get
        \begin{align}
            \|v_+\|^2 + \|v_-\|^2&\le 2f_+ + 2f_-\\
            &\le \frac{e^{2c''}}{1-\gamma(s)}\cdot 8\pi A^2\|\nabla v\|^2_{L^2}\cdot\left(e^{2(s-R_+)} + e^{2(R_--s)}\right)\\
            \label{v-bound-1}
            &\le A'^2\|\nabla v\|^2_{L^2}\cdot\left(e^{2(s-R_+)} + e^{2(R_--s)}\right)
        \end{align}
        with $A'$ denoting an absolute constant which may become bigger with subsequent appearances. Next, applying $\Pi_0$ to \eqref{cr-local} gives $v_0'(s) = \Pi_0(S(v)Lv)$, which implies
        \begin{align}
            |v_0'(s)|&\le c'|v|\cdot\|Lv\|\\
            &\le 2\pi c'A^2\eta\|\nabla v\|_{L^2}\left(e^{\frac12(s-R_+)} + e^{\frac12(R_--s)}\right)^2\\
            &\le A'\|\nabla v\|_{L^2}\left(e^{s-R_+} + e^{R_--s}\right)
        \end{align}
        and integrating this (recalling that $\langle v\rangle = 0$), we get the estimate
        \begin{align}\label{v-bound-2}
            |v_0(s)|^2\le A'^2\|\nabla v\|^2_{L^2}\cdot\left(e^{2(s-R_+)} + e^{2(R_--s)}\right).
        \end{align}
        Now combining \eqref{v-bound-1} and \eqref{v-bound-2} we find upper bounds for $\int_{[s-\frac14,s+\frac14]\times S^1}|v|^2\,ds\,dt$ for $s\in[R_-+\frac12,R_+-\frac12]$ and applying Lemma \ref{gs} gives the desired bound for $|(\nabla v)(s,t)|$ with a suitable absolute constant $H$ as in the statement.
    \end{proof}
	\end{section}
	%%%%%%%%%%%%%%%%%%%%%%%%%%%%%%%%%%%%%%%%%%%%%%%%%%
	%%%%%%%%%% HAUSDORFF DISTANCE %%%%%%%%%%%%%%%%%%%%
	%%%%%%%%%%%%%%%%%%%%%%%%%%%%%%%%%%%%%%%%%%%%%%%%%%
	\begin{section}{Hausdorff distance}\label{hdorff-appx}
	In this appendix, we recollect some properties of Hausdorff distances. Fix a complete metric space $(Z,d)$.
	\begin{definition}\label{hdorff-defn}
	    The \textbf{\emph{Hausdorff distance}} pseudo-metric $d_H$ on the power set $2^Z$ is defined as
	    \begin{align}\label{hdorff-distance-defined}
	        d_H(A,B) = \max\left\{\sup_{x\in A}\inf_{y\in B} d(x,y),\sup_{y\in B}\inf_{x\in A}d(x,y)\right\}
	    \end{align}
	    for any two subsets $A,B\subset Z$. Define $\cK(Z)$ to the space of compact subsets of $Z$ endowed with $d_H$. On the subset $\cK(Z)\subset 2^Z$, $d_H$ is a genuine metric.
	\end{definition}
	\begin{definition}\label{net-defined}
	    Given a constant $\gamma>0$ and a subset $Y\subset Z$, we call $Y$ an \textbf{\emph{$\gamma$-net for}} $Z$ if and only if
	    \begin{align}
	        d_H(Y,Z)<\gamma.
	    \end{align}
	    Define $\nu(Z,\gamma)$ to be the minimum possible cardinality $|Y|$ of a $\gamma$-net $Y\subset Z$. The metric space $Z$ is said to be totally bounded if $\nu(Z,\gamma)$ is finite for every choice of $\gamma > 0$.
	\end{definition}
	\begin{remark}\label{comp-tot-bd-cpt}
	    Recall that a metric space is compact if and only if it is complete and totally bounded. %For the rest of this appendix, let us fix a compact metric space $(Z,d)$.
	\end{remark}
	\begin{lemma}
	    Suppose $\{Z_i\subset Z\}_{i\in I}$ is a cover of $Z$ by finitely many subsets. Then, for any $\gamma>0$, we have
	    \begin{align}
	        \nu(Z,\gamma)\le\sum_{i\in I}\nu(Z_i,\gamma)
	    \end{align}
	\end{lemma}
	\begin{proof}
	    Given a $\gamma$-net $Y_i$ for each $Z_i$, we define $Y = \bigcup_i Y_i$. Then, $Y$ is a $\gamma$-net for $Z = \bigcup_i Z_i$.
	\end{proof}
	\begin{lemma}\label{unif-compact}
	    Let $\gamma>0$ and $K\subset Z$. Then, given any $\gamma$-net $Y\subset Z$, there exists a subset $Y'\subset Y$ and a $2\gamma$-net $K'$ for $K$ such that
	    \begin{align}
	        |K'|&\le |Y'| \\
	        d_H(K,Y')&<\gamma
	    \end{align}
	\end{lemma}
	\begin{proof}
	    Choose $d_H(Y,Z)<\gamma'<\gamma$. Define $Y'\subset Y$ to be the subset of points $y\in Y$ such that $d(y,K)<\gamma'$. For each $y\in Y'$, choose a point $x(y)\in K$ such that $d(y,x(y))<\gamma'$. Define $K' = \{x(y)\,|\,y\in Y'\}$. It is now immediate that $K'\subset K$ is a $2\gamma$-net and that $|K'|\le|Y'|$. Moreover, we have $d_H(K,Y')\le\gamma'<\gamma$.
	\end{proof}
	\begin{corollary}\label{K-net}
	    For any $\gamma>0$, we have the estimate
	    \begin{align}
	        \nu(\cK(Z),\gamma)\le 2^{\nu(Z,\gamma)}
	    \end{align}
	\end{corollary}
	\begin{proof}
	    Lemma \ref{unif-compact} shows that if $Y$ is a $\gamma$-net for $Z$, then the power set $2^Y$ is a $\gamma$-net for $\cK(Z)$.
	\end{proof}
	\begin{corollary}\label{unif-compact-final}
	    Let $\gamma>0$ and $K\subset Z$. Then, we have
	    \begin{align}
	        \nu(K,2\gamma)\le\nu(Z,\gamma)
	    \end{align}
	\end{corollary}
	\begin{proof}
	    We simply note that $|K'|\le|Y'|\le|Y|$ in Lemma \ref{unif-compact}.
	\end{proof}
	\begin{lemma}\label{hdorff-metric-complete}
	    $d_H$ is a complete metric on $\cK(Z)$.
	\end{lemma}
	\begin{proof}
	    Let $\{K_i\}_{i\ge 1}$ is a $d_H$-Cauchy sequence of compact subsets of $Z$ for which wish to construct a $d_H$-limit $K_\infty$. First, notice that the union $\hat K = \bigcup K_i$ is totally bounded. Indeed, given any $\gamma>0$, first choose $N = N(\gamma)$ such that $d_H(K_i,K_j)\le\frac12\gamma$ for $i,j\ge N$. For each $1\le i\le N$, let $Y_i$ be a finite $\frac12\gamma$-net for $K_i$. Then, $Y_1\cup\cdots\cup Y_N$ is a $\gamma$-net for $K$. Thus, replacing $Z$ by the closure of $\hat K$ (which is complete and totally bounded and thus, compact by Remark \ref{comp-tot-bd-cpt}), we may safely assume that $Z$ is compact.\\
	    \indent We can now ignore the fact that $\{K_i\}_{i\ge 1}$ is Cauchy and simply produce a convergent subsequence. Using Corollary \ref{unif-compact-final}, we see that there exists an increasing sequence of positive integers $\{N_n\}_{n\ge 1}$ and sequences $\{x_{i,j}\}_{j\ge 1}\subset K_i$ for each $i\ge 1$ with the following significance: for each integer $n\ge 1$, the collection $\{x_{i,j}\}_{j\le N_n}\subset K_i$ forms a $\frac1n$-net. By a diagonal argument, we can now pass to a subsequence, still denoted $\{K_i\}$, and construct a collection $\{x_j\}_{j\ge 1}\subset Z$ such that $x_j = \lim_{i\to\infty}x_{i,j}$ for each $i\ge 1$. Defining $K_\infty\subset Z$ to be the closure of $\{x_j\}_{j\ge 1}$, it easily follows that $\lim_{i\to\infty}d_H(K_i,K_\infty) = 0$.
	\end{proof}
	\begin{definition}
	    Suppose $T,Z,W$ are complete metric spaces and $\fX\subset T\times Z$ is closed subset, proper over $T$. For any $t\in T$, define $\fX_t\subset Z$ to be the inverse image of $\fX$ under the map $z\mapsto(t,z)$. Given a number $\Lambda>0$, define $\cF_\Lambda(\fX/T,W)$ to be the set of pairs $(t,f:\fX_t\to W)$ with $t\in T$ and $f$ being a continuous map with Lipschitz constant $\le\Lambda$. Define a metric on $\cF_\Lambda(\fX/T,W)$ by setting
	    \begin{align}
	        d((t,f),(t'f')) = \max\{d(t,t'), d_H(\Gamma_f,\Gamma_{f'})\}
	    \end{align}
	    where $\Gamma_f$ (resp. $\Gamma_{f'}$) is the graph of $f$ (resp. $f'$), viewed as a subset $Z\times W$. Here, the metric we use on $Z\times W$ is $d((z,w),(z',w')) = \max\{d(z,z'),d(w,w')\}$.
	\end{definition}
	\begin{lemma}\label{map-space-net}
	    If $\Lambda\ge 1$, then given any $\delta>0$, we can cover $\cF_\Lambda(\fX/T,W)$ by a collection $\{\cF_i\}_{i\in I_\delta}$ of subsets satisfying the conditions 
	    \begin{align}
	        \sup_{i\in I_\delta}\normalfont\text{diam}(\cF_i)&<4\delta \\
	        |I_\delta| = \nu(T,\delta)\cdot(1 &+ \nu(W,\delta))^{\nu(Z,\delta/\Lambda)}
	    \end{align}
	    and thus, for any subset $\cF\subset\cF_\Lambda(\fX/T,W)$, we have
	    \begin{align}
	        \nu(\cF,4\delta)\le |I_\delta|.
	    \end{align}
	\end{lemma}
	\begin{proof}
	    Let $A\subset T$, $C\subset W$ be $\delta$-nets and let $B\subset Z$ be a $(\delta/\Lambda)$-net. Choose $0<\gamma<\delta$ such that the previous sentence remains true when we replace $\delta$ by $\gamma$. Define $D = C\sqcup\{*\}$. Now, given any pair $(a,h)\in A\times D^B$ consisting of an element $a\in A$ and a set map $h:B\to D$, define the subset
	    \begin{align}
	        \cF_{(a,h)}\subset\cF_\Lambda(\fX/T,W)
	    \end{align}
	    as follows. It consists of elements $(t,f:\fX_t\to W)$ such that $d(t,a)\le\gamma$ and such that for each $b\in B$,
	    \begin{enumerate}[(i)]
	        \item if $h(b) = *$, then we have $d(b,\fX_t) > \gamma/\Lambda$ and,
	        \item if $h(b) = c\in C$, then there exists a point $b_t\in\fX_t$ such that $d(b,b_t)\le\gamma/\Lambda$ and $d(f(b_t),c)\le\gamma$.
	    \end{enumerate}
	    Since $A,B,C$ are $\gamma$,$(\gamma/\Lambda)$, $\gamma$-nets in $T,Z,W$ (resp.), it follows that the sets $\cF_{(a,h)}$ ranging over all possible $(a,h)$ cover the space $\cF_\Lambda(\fX/T,W)$. We will therefore be done if we estimate the diameter of each of these sets. Indeed, let $(t,f)$ and $(s,g)$ both lie in $\cF_{(a,h)}$. We note that
	    \begin{align}
	        d(t,s) \le d(t,a) + d(a,s) \le 2\gamma
	    \end{align}
	    and thus, we are left to show that $d_H(
	    \Gamma_f,\Gamma_g)\le 4\gamma$. Since the situation is symmetric, it will be enough to show that for any $x\in\fX_t$, we can find $y\in\fX_s$ such that $d(x,y)\le 4\gamma$ and $d(f(x),g(y))\le 4\gamma$. Starting with $x\in\fX_t$, first choose $b\in B$ such that $d(x,b)\le\gamma/\Lambda$. This is possible since $B\subset Z$ is a $(\gamma/\Lambda)$-net. It follows from (i) that $c = h(b)\in C$ and that we can find $b_t\in\fX_t$ as in (ii). Similarly, we can use (ii) to find $y\in\fX_s$ such that $d(b,y)\le\gamma/\Lambda$ and $d(g(y),c)\le\gamma$. Now, observe that
	    \begin{align}
	        d(x,y)&\le d(x,b) + d(b,y) \le 2\gamma/\Lambda\le 2\gamma \\
	        d(f(x),g(y))&\le d(f(x),c) + d(c,g(y)) \\
	        &\le d(f(x),f(b_t)) + d(f(b_t),c) + 
	        \gamma\\
	        &\le \Lambda\cdot d(x,b_t) + 2\gamma \\
	        &\le \Lambda\cdot(d(x,b) + d(b,b_t)) + 2\gamma \le 4\gamma
	    \end{align}
	    and this completes the proof.
	\end{proof}
	\end{section}
	%%%%%%%%%%%%%%%%%%%%%%%%%%%%%%%%%%%%%%%%%%%%%%%%%%
	%%%%%%%%%%%%%%%%% TREES %%%%%%%%%%%%%%%%%%%%%%%%%%
	%%%%%%%%%%%%%%%%%%%%%%%%%%%%%%%%%%%%%%%%%%%%%%%%%%
	\begin{section}{Trees and curves}\label{trees-appx}
	In this appendix, we establish the notation and some basic properties of (stable, marked) rooted trees. %\textcolor{red}
	{We also supply the proofs of Lemmas from \textsection\ref{families}.}
	%%%%%%%%%%%%%%%%%%%%%%%%%%%%%%%%%%%%%%%%%%%%
	\subsection{Definitions and basic properties}\label{tree-basics}
	\begin{definition}\label{graph-def}
	    A \emph{\textbf{graph}} (\emph{with half edges}) is a triple $(V,E,\partial)$ where $V$ is a finite set (with elements called \emph{\textbf{vertices}}), $E$ is a finite set (with elements called \emph{\textbf{edges}}) and a \emph{\textbf{boundary map}} $\partial:E\to 2^V$ such that for all $e\in E$, we have $1\le|\partial(e)|\le 2$. The elements of $\partial(e)$, for $e\in E$, are called the \emph{\textbf{endpoints}} of $e$. An edge $e\in E$ is called a \emph{\textbf{half}} or \emph{\textbf{full}} edge when it has $1$ or $2$ endpoints respectively. For a vertex $v\in V$, we define its \emph{\textbf{degree}} by {$\deg(v)$} $=|\{e\in E\,|\,v\in\partial(e)\}|$.
	\end{definition}
	\begin{definition}
	    Given two graphs $G = (V,E,\partial)$ and $G' = (V',E',\partial')$, an \emph{\textbf{isomorphism}} $\varphi:G\to G'$ is a pair of bijections $\varphi_V:V\to V'$ and $\varphi_E:E\to E'$ such that $2^{\varphi_V}\circ\partial = \partial'\circ\varphi_E$.
	\end{definition}
	\begin{definition}\label{edge-orient-def}
	    Given a graph $G = (V,E,\partial)$ and an edge $e\in E$, an \emph{\textbf{orientation on}} $e$ is an injective map $\fo_e:\partial(e)\to\{-1,+1\}$. An endpoint $v$ of $e$ is called \emph{\textbf{positive}} (resp. \emph{\textbf{negative}}), written as $v = e^+$ (resp. $v = e^-$), if $\fo_e(v) = +1$ (resp. $-1$). An \emph{\textbf{orientation}} of $G$ is a choice $\fo = \{\fo_e\}_{e\in E}$ of orientations for each the edges of $G$. When $G$ is equipped with an orientation $\fo$, for any vertex $v$, we use $v^+$ (resp. $v^-$) to denote the set of edges $e\in E$ for which $v\in\partial(e)$ and $\fo_e(v) = -1$ (resp. $\fo_e(v) = +1$).
	\end{definition}
	\begin{definition}\label{path-def}
	    A \emph{\textbf{path}} in a graph $G = (V,E,\partial)$ is defined to be a sequence $\cP = (v_1,\ldots,v_n)$ for some $n\ge 1$, where the $v_i$'s are (pairwise distinct) vertices such that for each $1\le i<n$, there exists $e_i\in E$ satisfying $\partial(e_i) = \{v_i,v_{i+1}\}$. In addition, if $G$ is oriented, then we say that $\cP$ is a \emph{\textbf{positive}} path if each $e_i$ can be chosen so that $v_i = e_i^-$.
	\end{definition}
	\begin{definition}\label{tree-def}
	    Given a graph $G=(V,E,\partial)$, we define its \emph{\textbf{geometric realization}} $\|G\|$ to be the $1$-dimensional simplicial complex with $0$-cells given by $V\sqcup E$ and a $1$-cell connecting $v\in V$ and $e\in E$ corresponding to each pair $(v,e)$ for which $v\in\partial(e)$. $G$ is called a \emph{\textbf{tree}} if and only if $\|G\|$ is contractible.
	\end{definition}
	\begin{definition}\label{rooted-tree-def}
	    A pair $(\cT,e_0)$ consisting of a graph $\cT$ and a half edge $e_0$ of $\cT$ is called a \emph{\textbf{rooted}} tree if the $\cT$ is a tree. We omit $e_0$ from the notation when it is clear from context. The unique endpoint $v_0$ of the edge $e_0$ is called the \emph{\textbf{root}} of the rooted tree, while $e_0$ itself is called the \emph{\textbf{root edge}}.
	\end{definition}
	\begin{lemma}[Orientation of a rooted tree]\label{orient-rooted-tree}
	    Any rooted tree $(\cT,e_0)$ has a unique orientation $\fo$ such that the endpoint of $e_0$ is positive and each vertex is the positive endpoint of a unique edge.
	\end{lemma}
	\begin{proof}
	   Given an edge $e\ne e_0$ of $\cT$, consider the unique path $v_0,\ldots,v_n$ along $\cT$ from $v_0$ to an endpoint $v_n$ of $e$ such that $v_i\not\in\partial(e)$ for all $0\le i<n$. We then set $\fo_{e}(v_n) = -1$. We now leave the reader to check that the orientation $\fo$ is well-defined and to verify its properties and uniqueness.
	\end{proof}
	\begin{definition}\label{common-ancestor}
	    Let $(\cT,e_0)$ be a rooted tree and $v,v'\in V$ be vertices of $\cT$. Then, the \emph{\textbf{nearest common ancestor}} of $v,v'$ is the unique vertex $u\in V$ such that there is a positive path $\cP$ (resp. $\cP'$) from $u$ to $v$ (resp. $v'$) and moreover, $u$ is the unique vertex common to $\cP$ and $\cP'$.
	\end{definition}
	\begin{definition}\label{stable-tree-def}
	    A tree $\cT$ is called \emph{\textbf{stable}} if each vertex $v$ of $\cT$ satisfies $\deg(v)\ge 3$. If $\deg(v) = 3$ for all vertices $v$, then we call $\cT$ a \emph{\textbf{maximal}} stable tree.
	\end{definition}
	\begin{lemma}\label{tree-edge-bound}
	    Suppose $\cT = (V,E,\partial)$ is a stable tree and let $E = \normalfont E_\text{ext}\sqcup E_\text{int}$ be the partition of the edge set into half and full edges. Then, we have the inequalities
	    \begin{align}
	        \normalfont
	        |E_\text{int}|+1 = |V|\le\frac13\cdot\displaystyle\sum_{v\in V}\deg(v)\le\normalfont|E_\text{ext}|-2
	    \end{align}
	    with equality precisely when $\cT$ is a maximal stable tree.
	\end{lemma}
	\begin{proof}
	    Counting the set $\{(e,v)\;|\;v\in\partial(e)\}\subset E\times V$ in two ways gives
	    \begin{align}
	        \sum_{v\in V}\deg(v) = |E_\text{ext}| + 2\cdot|E_\text{int}|.
	    \end{align}
	    Since $\cT$ is a tree, we see that $|V| = |E_\text{int}| + 1$. From stability, we get $\deg(v)\ge 3$ for all $v\in V$. Combining these three statements, we get the desired inequalities.
	\end{proof}
	\begin{definition}\label{mark-def}
	    A pair $(G,E_0)$ is called a \emph{\textbf{marked}} graph if $G = (V,E,\partial)$ is a graph and $E_0\subset E$ is a subset consisting entirely of half edges. We call $E_0$ the \emph{\textbf{marking}} in this case and omit it from the notation when it is clear from context. More generally, if $S$ is a finite set, then an injective map $\iota:S\to E$ (with image lying in the set of half edges) is called an \emph{\textbf{$S$-marking}} of $G$.
	\end{definition}
	\begin{definition}\label{splitting-def}
	    Consider a tree $\cT = (V,E,\partial)$. Let $F\subset E$ be a set of full edges. Define a new graph $\cT'$, called the \emph{\textbf{splitting of $\cT$ along $F$}}, on the same set $V$ of vertices as follows. The edge set of $\cT'$ is given by
	    \begin{align}
	        E' = (E\setminus F)\cup\{(e,v)\;|\;e\in F\,,v\in\partial(e)\}
	    \end{align}
	    while the boundary map $\partial'$ is given by $\partial'(e,v) = \{v\}$ and $\partial'|_{E\setminus F} = \partial_{E\setminus F}$. Then, the graph $\cT'$ has $|F|+1$ connected components and each is a tree.
	\end{definition}
	\begin{remark}
	    In the situation of Definition \ref{splitting-def}, for all $v\in V$, the degree of $v$ is the same in both $\cT$ and $\cT'$. In particular, if $\cT$ is stable then so is each connected component of $\cT'$. 
	\end{remark}
	\begin{definition}
	    Suppose $(\cT,e_0)$ is a rooted tree, $E_0\subset E$ is a marking and $F\subset E$ is a set of full edges. Define a marking $E_0'$ on $\cT'$, the splitting of $\cT$ along $F$, by setting
	    \begin{align}
	        E_0' = E_0\cup\{(e,v)\;|\;e\in F\,,v\in\partial(e)\}.
	    \end{align}
	    Moreover, define $R = \{e_0\}\cup\{(e,e^+)\;|\;e\in F\}\subset E'$, where we are using the orientation $\fo$ on $(\cT,e_0)$ provided by Lemma \ref{orient-rooted-tree}. Then, each connected component of $\cT'$ contains a unique edge in $R$, which provides it with the structure of a rooted tree. Moreover, equipped with the marking inherited from $E'_0$, we call this the collection of \emph{\textbf{rooted marked trees obtained by splitting $(\cT,e_0,E_0)$ along $F$}}. 
	\end{definition}
	\begin{lemma}
	    For an integer $n\ge 2$, define $T_n$ to be the number of stable rooted trees, up to isomorphism preserving the root edge, with $n+1$ half edges. We then have the bound
	    \begin{align}
	        T_n\le\frac1{4\sqrt 3}\left(\frac{2^n}{\sqrt n}\right)^3
	    \end{align}
	    for all integers $n\ge 2$.
	\end{lemma}
	\begin{proof}
	    For $n\ge 2$, define $T'_n$ to be the number of maximal stable rooted trees, up to isomorphism preserving the root, with $n+1$ half edges. We then have the inequality
	    \begin{align}
	        T_n\le 2^{n-2}T'_n
	    \end{align}
	    which follows by observing that any stable rooted tree can be formed by contracting a subset of full edges in some maximal stable tree. $T'_n$ can alternatively be described as the number of ways of multiplying $n$ variables using a binary operation which is commutative but non-associative. This number is clearly bounded above by the number ways of bracketing the expression $x_1\cdots x_n$ (where the order of the $n$ variables is fixed and the product is non-commutative and non-associative), the Catalan number $C_n = \frac{1}{n+1}\binom{2n}{n}$. Combining this with the elementary estimate $\binom{2n}{n}\le \frac{4^n}{\sqrt{3n+1}}$ (\cite[Solution to Problem 1]{kaz-ineq}), we get the stated result.
	\end{proof}
	%%%%%%%%%%%%%%%%%%%%%%%%%%%%%%%%%%%%%%%%
	\subsection{Curves modeled on trees}\label{curve-family-proofs}
	We prove the Lemmas from \textsection\ref{curve-family}.
	\begin{proof}[Proof of Lemma \ref{prop-family-curve}]
	    \begin{enumerate}[(i)]
	        \item This follows from \eqref{def-eqn} since $\rho^*_{u,e}\ne 0$ by assumption. Indeed, it is clear that \eqref{def-eqn} expresses $[x_u:y_u]$ as a one-to-one function of $[x_v:y_v]$ and vice versa. Since $\cT$ is connected, the result follows.
	        \item Equation \eqref{def-eqn} shows that if $[x_u:y_u] = [1:0]$, then we also have $[x_v:y_v] = [1:0]$. In fact, we can solve \eqref{def-eqn} to get $[x_v:y_v]$ as a function of $[x_u:y_u]$ as below    \begin{align}\label{sol-near-infty}
	                [x_v:y_v] = [x_u - z^*_{u,e}y_u:\gamma^*_e\rho^*_{u,e}y_u]  
	        \end{align}
	        for $[x_u:y_u]$ Zariski close to $[1:0]$. Thus, $\cC_{p^*,\cT}$ is the graph of a map $\bP^1\to\prod_{v\ne v_0}\bP^1$ near $[1:0]$.
	        \item Looking at equation \eqref{def-eqn}, it is clear that for $y_v\ne 0$, we can express $[x_u:y_u]$ as a function of $[x_v:y_v]$ as follows
	        \begin{align}
	        \label{marked-point-descent}
	            [x_u:y_u] = [z^*_{u,e}y_v + \gamma^*_e\rho^*_{u,e}x_v:y_v]
	        \end{align}
	        with $y_v\ne 0$ again. Also notice that if $\gamma^*_e\ne 0$, then we are in the situation of the proof of (i), i.e., $[x_v:y_v]$ is also a function of $[x_u:y_u]$.\\\\
	        Define $A_w\subset V$ to be the set of vertices appearing in the (positive) path from $v_0$ to $w$. Let $B_{w,p^*}\subset V$ be the set of vertices $w'$ such that there exists a vertex $w''\in A_w$ and a positive path from $w''$ to $w'$ meeting $A_w$ only in $w''$ and traversing only edges $e$ for which $\gamma^*_e\ne 0$. Our discussion so far implies that for $w'\in B_{w,p^*}$, we can solve for $[x_{w'}:y_{w'}]$ uniquely as a function of $[x_w:y_w]$ at least near $[z^*_{w,f}:1]$. Moreover, the formula \eqref{marked-point-descent} shows that when $w''\in A_w$ and $[x_w:y_w] = [z^*_{w,f}:1]$, we have $[x_{w''}:y_{w''}] = [z^{w''}_{w,f}(p^*):1]$. Now, suppose $u\in B_{w,p^*}$ and $e\in E$ with $e^- = u$, $e^+ = v$ and $\gamma^*_e = 0$. Then, we can solve \eqref{def-eqn} uniquely for $[x_v:y_v]$ in terms of $[x_u:y_u]$ (to get $[x_v:y_v] = [1:0]$) except at $[x_u:y_u] = [z^*_{u,e}:1]$. If $[x_v:y_v]$ is determined to be $[1:0]$, then the argument of (ii) shows that $[x_{v'}:y_{v'}] = [1:0]$ also holds for every $v'$ such that there is a positive path from $v$ to $v'$.\\\\
	        \indent Thus, to show that $[x_v:y_v]$ is uniquely determined (and equal to $[1:0]$) in terms of $[x_w:y_w]$ near $[z^*_{w,f}:1]$, we only need to show that there is no point on $q\in\cC_{p^*,\cT}$ for which $\pi_u(q) = [z^*_{u,e}:1]$ and $\pi_w(q) = [z^*_{w,f}:1]$, with $(u,e)$ as in the previous paragraph. Indeed, if there were such a point $q$, take $w''\in A_w$ to be the nearest common ancestor of $u$ and $w$. From $\pi_u(q) = [z^*_{u,e}:1]$, we get $\pi_{w''}(q) = [z^{w''}_{u,e}(p^*):1]$ while from $\pi_w(q) = [z^*_{w,f}:1]$, we get $\pi_{w''}(q) = [z^{w''}_{w,f}(p^*):1]$. This gives $z^{w''}_{u,e}(p^*) = z^{w''}_{w,f}(p^*)$, which contradicts the condition $F_\cT(p^*)\ne 0$. This completes the proof that $\pi_w:\cC_{p^*,\cT}\to\bP^1$ is a local isomorphism over $[z^*_{w,f}:1]$. 
	        \item By induction, assume that $\cC_{p^*_u,\cT_u}$ and $\cC_{p^*_v,\cT_v}$ are both prestable genus zero curves. It's clear from the defining equations that $\cC_{p^*,\cT}$ is the intersection of $\cC_{p^*_u,\cT_u}\times\cC_{p^*_v,\cT_v}$ with the inverse image of
	        \begin{align}
	             (\bP^1_u\times[1:0])\cup([z^*_{u,e}:1]\times \bP^1_v)\subset\bP^1\times\bP^1
	        \end{align}
	        under the $(u,v)$-coordinate projection $(\bP^1)^V\to\bP^1\times\bP^1$. Now, we can use (ii) to construct $q_{p^*,u}$ and (iii) to construct $q_{p^*,v}$ and \eqref{break-into-comps} follows. Moreover, by (ii) and (iii), $q_{p^*,u}\in\cC_{p^*_v,\cT_v}$ and $q_{p^*,v}\in\cC_{p^*_u,\cT_u}$ are smooth points and thus, $\cC_{p^*,\cT}$ is also a prestable genus $0$ curve.
	        \item We can deduce the existence and uniqueness of the maps $\sigma_\cT(e,\cdot)$ for each $e\in E_\text{ext}$ using (ii) (resp. (iii)) for $e = e_0$ (resp. $e\ne e_0$). In fact, (ii) and (iii) imply that $\sigma_\cT(e,p^*)\in\cC_{p^*,\cT}$ is a smooth point for each $p^*\in\cM_\cT$ and $e\in E_\text{ext}$. We are just left to prove that $\sigma_\cT$ is an embedding. Suppose to the contrary that we have $e,e'\in E_\text{ext}$ with $e\ne e'$ and $\sigma_\cT(e,p^*) = \sigma_\cT(e',p^*)$. In view of the proof of (ii), we must have $e\ne e_0$ and $e'\ne e_0$. Thus, let $v = e^-$, $v' = e'^-$ and define $u\in V$ be the nearest common ancestor of $v,v'$. Arguing as in the proof of (iii) gives
	        \begin{align}
	           \pi_u(\sigma_\cT(e,p^*)) &= [z^u_{v,e}(p^*):1]\\
	                \pi_u(\sigma_\cT(e',p^*)) &= [z^u_{v',e'}(p^*):1]
	       \end{align}
	       and thus, $\sigma_\cT(e,p^*) = \sigma_\cT(e',p^*)$ gives $z^u_{v,e} = z^u_{v',e'}$ at $p^*$, a contradiction to $F_\cT(p^*)\ne 0$.
	    \end{enumerate}
	\end{proof}
	\begin{proof}[Proof of Lemma \ref{compact-curve-prop}]
	    \begin{enumerate}[(i)]
	        \item By \eqref{M-def1} and \eqref{M-def2}, we have $|z^*_{v,e}| + |\rho^*_{v,e}|\le 3\theta\le\frac12$ for any $e\in v^+$, which proves the first assertion. Now, take $c<\tau^{-1}$ and $e\ne e'$ in $v^+$ and note that by \eqref{M-def3}, we have
	        \begin{align}
	            |z^*_{v,e}-z^*_{v,e'}|> c\cdot(|\rho^*_{v,e}|+|\rho^*_{v,e'}|)
	        \end{align}
	        which proves the second assertion.
	        \item This is a quantitative version of Lemma \ref{prop-family-curve}(ii). Indeed, it suffices to observe that in formula \eqref{sol-near-infty}, if $|x_u|\ge|y_u|$, then we also have $|x_v|\ge|y_v|$. Indeed, \eqref{sol-near-infty} gives
	        \begin{align}
	            \left|\frac{x_v}{y_v}\right| = \left|\frac{\frac{x_u}{y_u} - z^*_{u,e}}{\gamma^*_e\rho^*_{u,e}}\right|\ge\frac{1-\theta}{2\tau\theta}>1
	        \end{align}
	        whenever we have $|x_u|\ge|y_u|$.
	        \item This is a quantitative version of Lemma \ref{prop-family-curve}(iii). Making use of the notation from that proof, we note that for $[x_w:y_w]$ such that $|\frac{x_w}{y_w} - z^*_{w,f}|\le|\rho^*_{w,f}|$, the corresponding value of $[x_{w'}:y_{w'}]$ for any $w\ne w'\in A_w$ satisfies $|\frac{x_{w'}}{y_{w'}} - z^*_{w',f'}|\le|\rho^*_{w',f'}|$, where $f'$ is the first edge traversed in the (positive) path from $w'$ to $w$. Indeed, we can see from \eqref{marked-point-descent} that if $e'$ is an edge with $e'^- = v$, then
	        \begin{align}
	            \left|\frac{x_u}{y_u} - z^*_{u,e}\right| = \left|\gamma^*_e\rho^*_{u,e}\frac{x_v}{y_v}\right|\le \tau|\rho_{u,e}^*|\cdot(|z^*_{v,e'}| + |\rho^*_{v,e'}|)\le3\theta\tau|\rho^*_{u,e}|\le|\rho^*_{u,e}|
	        \end{align}
	        whenever $|\frac{x_v}{y_v} - z^*_{v,e'}|\le|\rho^*_{v,e'}|$. In particular, $|z^{w'}_{w,f}(p^*)-z^*_{w',f'}|\le|\rho^*_{w',f'}|$. Now, following the argument of Lemma \ref{prop-family-curve}(iii), we see that we are done once we observe that the balls $B^2(z^*_{w',f'},|\rho^*_{w',f'}|)$ and $B^2(z^*_{w',f''},|\rho^*_{w',f''}|)$ are disjoint by (i) above, for $w'\in A_w$ and $f'\ne f''$ in $(w')^+$.
	        \item By projecting $\cC_{p^*,\cT}$ to $\bP^1\times\bP^1$ using $\pi_u\times\pi_v$ it is clear that we only need to perform the identification \eqref{local-plumbing} over the regions
	        \begin{align}
	            \left\{|\gamma^*_e|\le \left|z_u\right|\le1\right\}&\subset D_u\\
	            \left\{|\gamma^*_e|\le\left|z'_v\right|\le1\right\}&\subset D_v
	        \end{align}
	        to obtain $\cC_{p^*,\cT}$ from $\cC_{p^*_u,\cT_u}$ and $\cC_{p^*_v,\cT_v}$.
	    \end{enumerate}
	\end{proof}
	\begin{lemma}\label{disc-round-to-flat}
	    The holomorphic embedding $\iota:B^2(1)\to\bP^1$ given by $z\mapsto[z:1]$ satisfies $1\le\frac{\iota^*\omega_{\bP^1}}{\omega_\bC}\le 4$ pointwise on $B^2(1)$.
	\end{lemma}
	\begin{proof}
	    This is a direct consequence of a computation using the formulas for $\omega_\bC$ and $\omega_{\bP^1}$ given in \textsection\ref{notation}.
	\end{proof}
	\begin{lemma}\label{annulus-cyl-flat}
	    Let $R>0$ and define $\delta = e^{-R}$ and 
	    \begin{align}
	        A_\delta = \{(z,w)\in B^2(1)\times B^2(1)\;|\;zw=\delta^2\}\subset\bC^2.
	    \end{align}
	    Define the isomorphism $\nu:[-R,R]\times S^1\to A_\delta$ by $\nu(s,t) = (e^{-R}e^{-(s+it)},e^{-R}e^{s+it})$. We then have
	    \begin{align}
	        \nu^*(\omega_\bC\oplus\omega_\bC) = e^{-2R}(e^{2s} + e^{-2s})\,ds\wedge dt.
	    \end{align}
	\end{lemma}
	\begin{proof}
	    Direct computation using the formula for $\omega_\bC$ given in \textsection\ref{notation}.
	\end{proof}
	\begin{lemma}\label{annulus-flat-dist}
	    Let $0\le\delta<1$ and suppose $z,w,z',w'\in\bC$ with $|z|,|w|,|z'|,|w'|
	    \le 1$ and $zw = z'w' = \delta^2$. Then, there exist piecewise $C^1$ paths $\gamma_z,\gamma_w:[0,1]\to\bC$ from $z$ to $z'$, resp. $w$ to $w'$, such that
	    \begin{align}
	        \sup_{t\in[0,1]}|\gamma_z(t)|\le 1\\
	        \sup_{t\in[0,1]}|\gamma_w(t)|\le 1\\
	        \gamma_z\gamma_w \equiv \delta^2
	    \end{align}
	    and we have the length estimate $\ell(
	    \gamma_z) + \ell(\gamma_w)\le 8\pi\max\{|z - z'|,|w-w'|\}$.
	\end{lemma}
	\begin{proof}
	    First, note that given any two real numbers $0<R\le R'$, and $a,b\in\bC$ with $R\le|a|,|b|\le R'$, we can join $a$ to $b$ by a path lying in $\{\zeta\in\bC\;|\;R\le|\zeta|\le R'\}$ with length $\le\frac12\pi|a-b|$. Indeed, we can simply take the straight line path from $a$ to $b$ and, if it intersects the disk $B^2(R)$, replace this portion of the path by the smaller arc along $S^1(R)$. When $\delta = 0$, the lemma is obvious, so assume $\delta\ne 0$.\\\\
	    Without loss of generality, assume we have $|z|\ge\delta$ and $|w|\le\delta$. We will consider the following cases.
	    \begin{enumerate}[(i)]
	        \item Suppose $|z'|\ge\delta$ and $|w'|\le\delta$. Choose a path $\gamma_z$ from $z$ to $z'$ such that $\delta\le\inf_t|\gamma_z(t)|\le\sup_t|\gamma_z(t)|\le 1$ and has length $\le\frac12\pi|z-z'|$. Define $\gamma_w = \delta^2/\gamma_z$. Note that we have $|\dot\gamma_w| = \frac{\tau^2}{|\gamma_z|^2}|\dot\gamma_z|\le|\dot\gamma_z|$, and so, $\ell(\gamma_w)\le\ell(\gamma_z)$ in this case. As a result, $\ell(\gamma_z) + \ell(\gamma_w)\le\pi|z-z'|$.
	        \item Suppose $|z'|\le\delta$ and $|w'|\ge\delta$. We consider the following sub-cases.
	        \begin{enumerate}[(a)]
	            \item Suppose $|z'|\le\frac12|z|$ and $\frac12|w'|\ge |w|$. We then have $|z - z'|\ge\frac12|z|$ and $|w - w'|\ge\frac12|w'|$. Let $\gamma_{1,z}$ be a path from $z$ to $\delta$ with $\delta\le\inf_t|\gamma_{1,z}(t)|\le\sup_t|\gamma_{1,z}(t)|\le 1$ and $\ell(\gamma_{1,z})\le\frac12\pi|z - \delta|$. By the argument from case (i), we find that $\gamma_{1,w} = \delta^2/\gamma_{1,z}$ has length $\le\ell(\gamma_{1,z})$. Thus,
	            \begin{align}
	                \ell(\gamma_{1,z}) + \ell(\gamma_{1,w})\le\pi|z-\delta|\le 2\pi|z|\le 4\pi|z - z'|.
	            \end{align}
	            Replace $z$ by $w'$ in the above to obtain paths $\gamma_{2,w}$ (from $\delta$ to $w'$) and $\gamma_{2,z} = \delta^2/\gamma_{2,w}$. We can now take $\gamma_z$ to be the concatenation of $\gamma_{1,z}$ and $\gamma_{2,z}$. Defining $\gamma_w$ similarly, we find
	            \begin{align}
	                \ell(\gamma_z) + \ell(\gamma_w)\le8\pi\max\{|z-z'|,|w-w'|\}.
	            \end{align}
	            \item Suppose that $\frac12|z|\le|z'|$ and $\frac12|w'|\le|w|$. This implies that $\frac\delta2\le|z'|\le\delta\le|z|\le 2\delta$. Now choose any path $\gamma_z$ from $z$ to $z'$ such that $\max\{\frac12\delta,\delta^2\}\le\inf_t|\gamma_z(t)|\le\sup_t|\gamma_z(t)|\le\min\{2\delta,1\}$ with length $\le\frac12\pi|z - z'|$. Again arguing as in case (i), we see that $\gamma_w = \delta^2/\gamma_z$ has length $\le 4\ell(\gamma_z)$. Thus, we find $\ell(\gamma_z) + \ell(\gamma_w)\le\frac52\pi|z - z'|$.
	        \end{enumerate}
	    \end{enumerate}
	\end{proof}
	\end{section}
	%%%%%%%%%%%%%%% REFERENCES %%%%%%%%%%%%%%%%%%%%%%%
	\addcontentsline{toc}{section}{References}
	\bibliographystyle{amsalpha}
	\bibliography{main}

\providecommand{\bysame}{\leavevmode\hbox to3em{\hrulefill}\thinspace}
\providecommand{\MR}{\relax\ifhmode\unskip\space\fi MR }
\providecommand{\MRhref}[2]{%
  \href{http://www.ams.org/mathscinet-getitem?mr=#1}{#2}
}
\providecommand{\href}[2]{#2}
\begin{thebibliography}{{Gro}13}

\bibitem[ACG11]{ACGH2}
E.~Arbarello, M.~Cornalba, and P.~A. Griffiths, \emph{Geometry of algebraic
  curves. {V}olume {II}}, Grundlehren der Mathematischen Wissenschaften
  [Fundamental Principles of Mathematical Sciences], vol. 268, Springer,
  Heidelberg, 2011, With a contribution by Joseph Daniel Harris. \MR{2807457}

\bibitem[FO99]{FO-kuranishi}
K.~Fukaya and K.~Ono, \emph{Arnold conjecture and {G}romov-{W}itten invariant},
  Topology \textbf{38} (1999), no.~5, 933--1048. \MR{1688434}

\bibitem[Gro85]{gromov85}
M.~Gromov, \emph{Pseudo holomorphic curves in symplectic manifolds}, Invent.
  Math. \textbf{82} (1985), no.~2, 307--347. \MR{809718}

\bibitem[{Gro}13]{groman}
Y~{Groman}, \emph{{A thick-thin decomposition of $J$-holomorphic curves}},
  arXiv e-prints (2013), arXiv:1311.7564.

\bibitem[HT09]{HT-obg2}
M.~Hutchings and C.~H. Taubes, \emph{Gluing pseudoholomorphic curves along
  branched covered cylinders. {II}}, J. Symplectic Geom. \textbf{7} (2009),
  no.~1, 29--133. \MR{2491716}

\bibitem[Hum97]{hummel}
C.~Hummel, \emph{Gromov's compactness theorem for pseudo-holomorphic curves},
  Progress in Mathematics, vol. 151, Birkh\"{a}user Verlag, Basel, 1997.
  \MR{1451624}

\bibitem[Kaz61]{kaz-ineq}
N.~D. Kazarinoff, \emph{Geometric inequalities}, New Mathematical Library,
  vol.~4, Random House, New York-Toronto, 1961. \MR{0130134}

\bibitem[MS12]{McSa}
D.~McDuff and D.~Salamon, \emph{{$J$}-holomorphic curves and symplectic
  topology}, second ed., American Mathematical Society Colloquium Publications,
  vol.~52, American Mathematical Society, Providence, RI, 2012. \MR{2954391}

\bibitem[Par16]{ParVFC}
J.~Pardon, \emph{An algebraic approach to virtual fundamental cycles on moduli
  spaces of pseudo-holomorphic curves}, Geom. Topol. \textbf{20} (2016), no.~2,
  779--1034. \MR{3493097}

\end{thebibliography}
\end{document}